\documentclass[a4paper,10pt]{amsart}
\usepackage{amsmath, amssymb, amsthm}
\usepackage{verbatim}
\usepackage{hyperref}

\newcounter{theorem}
\newtheorem{thm}[theorem]{Theorem}
\newtheorem{lemma}[theorem]{Lemma}
\newtheorem{prop}[theorem]{Proposition}
\newtheorem{cor}[theorem]{Corollary}
\newtheorem{defn}[theorem]{Definition}

\theoremstyle{remark}
\newtheorem*{remark*}{Remark}
\newtheorem{remark}[theorem]{Remark}

\numberwithin{equation}{section}
\numberwithin{theorem}{section}

\newcommand{\e}{\epsilon}
\newcommand{\dl}{\delta}

\newcommand{\R}{\mathbb{R}}
\newcommand{\C}{\mathbb{C}}
\newcommand{\N}{\mathbb{N}}

\renewcommand{\setminus}{\backslash}
\newcommand{\K}{\mathcal{K}}
\newcommand{\tens}{\otimes}
\newcommand{\dsum}{\oplus}
\newcommand{\bigdsum}{\bigoplus}

\newcommand{\chii}{\raisebox{2pt}{$\chi$}}
\renewcommand{\emptyset}{\varnothing}
\newcommand{\iso}{\cong}

\newcommand{\Cu}{\mathcal{C}u}
\newcommand{\jsZ}{\mathcal{Z}}

\newcommand{\supp}{\mathrm{supp}\,}

\newcommand{\Prim}{\mathrm{Prim}}
\newcommand{\Ped}{\mathrm{Ped}\,}
\newcommand{\Span}{\mathrm{span}\,}
\newcommand{\id}{\mathrm{id}}
\newcommand{\dn}{\mathrm{dim}_{nuc}}
\newcommand{\her}[1]{\mathrm{her}\left(#1\right)}

\newcommand{\labelledthing}[2]{\hspace{4pt}\buildrel {#2} \over #1 \hspace{3pt}} 

\newcommand{\labelledrightarrow}{\labelledthing{\longrightarrow}}

\newcommand{\ccite}[2]{\cite[#1]{#2}}
\newcommand{\alabel}{\label}
\newcommand{\TODO}[1]{(\textbf{To do: \textit{#1}})}

\begin{document}

\title[Stably projectionless $C^*$-algebras]{Nuclear dimension, $\jsZ$-stability, and algebraic simplicity for stably projectionless $C^*$-algebras}
\author{Aaron Tikuisis}
\address{\hskip-\parindent
Aaron Tikuisis, Mathematisches Institut der WWU M\"unster, Einsteinstra\ss{}e 62, 48149, M\"unster, Germany.}
\email{a.tikuisis@uni-muenster.de}


\thanks{This research was supported by DFG (SFB 878).}

\begin{abstract}
The main result here is that a simple separable $C^*$-algebra is $\jsZ$-stable (where $\jsZ$ denotes the Jiang-Su algebra) if (i) it has finite nuclear dimension or (ii) it is approximately subhomogeneous with slow dimension growth.
This generalizes the main results of \cite{Toms:rigidity,Winter:pure} to the nonunital setting.
As a consequence, finite nuclear dimension implies $\jsZ$-stability even in the case of a separable $C^*$-algebra with finitely many ideals.
Algebraic simplicity is established as a fruitful weakening of being simple and unital, and the proof of the main result makes heavy use of this concept. 
\end{abstract}

\maketitle

\section{Introduction}

The program to classify unital, simple, separable, nuclear $C^*$-algebras has recently seen a small paradigm shift \cite{ElliottToms}.
In light of the fact that there exist such $C^*$-algebras that cannot be classified by ordered $K$-theory and traces \cite{Toms:annals}, the current trend is to try to identify, using regularity properties, the $C^*$-algebras which are (or should be) classifiable.
This idea is crystallized in a conjecture (cf.\ \ccite{Section 3.5}{TomsWinter:V1}), due to Andrew Toms and Wilhelm Winter, that among these $C^*$-algebras, the following properties are equivalent:
\begin{enumerate}
\item \alabel{Intro-dn} finite nuclear dimension (defined in \cite{WinterZacharias:NucDim});
\item \alabel{Intro-Z} $\jsZ$-stability (being isomorphic to one's tensor product with the Jiang-Su algebra $\jsZ$, introduced in \cite{JiangSu}); and
\item \alabel{Intro-unperf} almost unperforated Cuntz semigroup.
\end{enumerate}
%Their conjecture goes on to say that the class of such algebras which share these properties is classifiable by ordered $K$-theory and traces.
One might also add that, modulo the UCT, conditions \ref{Intro-dn}-\ref{Intro-unperf} should be equivalent to being classifiable (though in this form, such a statement is not well-formed because classifiability is a property of a class of $C^*$-algebras, and not a property of a single $C^*$-algebra).
%\TODO{cite: Elliott, Rordam - classification of certain infinite simple $C^*$-algebras, II}
%However, the intuition expressed in the last statement can be formalized by the following equivalent phrasing of Toms and Winter's conjecture (the fact that it is equivalent is due to range-of-invariant results \TODO{cite-PI case}\cite{ElliottRordam:pirange,Elliott:ashrange}):
%
%\begin{conjecture}
%Let $A$ be a unital, simple, separable, nuclear $C^*$-algebra.
%Then (i)-(iii) are equivalent for $A$, as is
%\begin{enumerate}
%\item[(iv)] $A$ is either a Kirchberg algebra in the UCT class, or is isomorphic to the inductive limit of $2$-dimensional non-commutative CW complex (as defined in \ccite{Definition 2.4.2}{EilersLoringPedersen:NCCW}).
%\end{enumerate}
%\end{conjecture}
%Phrased this way, the power of the conjecture in understanding the structure of $C^*$-algebras is immediately clear.
%
%Many partial verifications of this conjecture have been made \cite{Lin:AsympClassification,MatuiSato:Comp,Rordam:Z,Winter:drSH,Winter:localizing,Winter:drZstable}.

At the same time, certain examples of simple, stably projectionless $C^*$-algebras have emerged --- including certain crossed products of $\mathcal{O}_2$ by $\mathbb{R}$ \cite{KishimotoKumjian:StablyProj}, a nuclear, separable, non-$\jsZ$-stable example \ccite{Theorem 4.1}{ProjlessReg}, and others \cite{Elliott:ashrange,Jacelon:R,Tsang:range}.
Certain tools, old and new, already allow one to understand some of the structure of these algebras under special hypotheses \cite{Connes:ThomR,Robert:NCCW}.
However, most of the theory on regularity properties for $C^*$-algebras was developed only in the unital case, and leads one to ask what obstructions (if any) exist with nonunital algebras.
Certainly, the unital case carries with it a number of simplifications: for instance, the simplex of traces on a $C^*$-algebra which take the value $1$ at the unit (i.e.\ which are states) plays an indispensible role in the theory, and it is far from obvious what the correct replacement is in the nonunital case.
Nonetheless, one hopes that the simplifications in the unital case are superficial, and that ultimately, one can find and prove analogues or generalizations of the known unital results.

The main results of this article are substantiations of this hope.
They are a generalization to the nonunital case of the breakthrough results of \cite{Winter:pure}, namely that \ref{Intro-dn} implies \ref{Intro-Z}, and that \ref{Intro-unperf} together with almost divisible Cuntz semigroup and locally finite nuclear dimension imply \ref{Intro-Z}.
It follows that \ref{Intro-dn} continues to imply \ref{Intro-Z} in the case of $C^*$-algebras with finitely many ideals.

Algebraic simplicity plays a major role in this article.
Simple unital $C^*$-algebras are automatically algebraically simple (this is a consequence of the fact that elements close enough to the unit are invertible).
Algebraically simple algebras therefore extend the class of unital simple algebras, and these algebras seem to retain features that allow certain regularity properties from the unital theory (especially those involving traces) to be phrased in a fruitful way.
We show that every simple $C^*$-algebra is Morita equivalent to an algebraically simple one --- and therefore, algebraic simplicity is not itself so much a regularity property, but rather a tool for helping to analyze any simple $C^*$-algebra.
(Note that the regularity properties appearing in Toms and Winter's conjecture are all preserved under Morita equivalence.)

The result that simple, nuclear $C^*$-algebras whose Cuntz semigroups are almost unperforated and almost divisible are $\jsZ$-stable has an important consequence for approximately subhomogeneous $C^*$-algebras (using a main result of \cite{TT:ranks}).
Namely, it allows us to characterize slow dimension growth for these algebras as equivalent to $\jsZ$-stability.
In the unital case, this result is obtained as a combination of results by Toms \cite{Toms:rigidity} and Winter \cite{Winter:pure}.
Many of the motivating examples of stably projectionless $C^*$-algebras are known to be approximately subhomogeneous \cite{Elliott:ashrange,Jacelon:R,ProjlessReg,Tsang:range}.
In fact, the class of approximately subhomogeneous, stably projectionless algebras is known to even include some crossed products of $\mathcal{O}_2$ by $\R$ \cite{Dean:CtsFld}.

It should be noted that Norio Nawata has, in \ccite{Theorem 5.11}{Nawata:Projless}, generalized to the nonunital setting a recent $\jsZ$-stability theorem of Matui and Sato \ccite{Theorem 1.1}{MatuiSato:Comp}.
Although they were found independently, many of the Nawata's innovations are similar to those found here.

The organization of this article is as follows.
Traces are absolutely indispensible to the results of this article; these are introduced in Section \ref{AlgSimpleSec} in connection with algebraically simple $C^*$-algebras, for whom analysis involving traces is most tractable.
Properties closely related to traces also appear in Section \ref{AlgSimpleSec}.
We recall the definition of nuclear dimension in Section \ref{DimNucSec}, together with the supporting concept of order zero maps (which will reappear in a characterization of $\jsZ$-stability).
Section \ref{AsympSeqSec} concerns another extremely useful tool, the asymptotic sequence algebra.
This algebra is heavily used throughout the article, largely to compress approximate notions (e.g.\ approximate commutation) and make proofs more conceptual.
We recall what the Jiang-Su algebra is, and provide a characterization of when (nonunital) algebras are $\jsZ$-stable in Section \ref{ZSec} (most of this section actually applies to strongly self-absorbing $C^*$-algebras in place of $\jsZ$).

The remaining sections contain the steps of the proofs of the main results.
As in \cite{Winter:pure}, both main results stated above are reduced to showing that algebras with locally finite nuclear dimension, $m$-almost divisible Cuntz semigroup and $m$-comparison (see Definition \ref{CompDivDefn}) are $\jsZ$-stable.
Particularly, Section \ref{DimNucDivSec} shows that finite nuclear dimension implies $m$-almost divisible Cuntz semigroup for some $m$, while \cite{Robert:dimNucComp} shows that it implies $m$-comparison.
It is proven in Section \ref{DivCentralSec} that if a $C^*$-algebra has locally finite nuclear dimension, then $m$-almost-divisibility implies $0$-almost-divisibility, even in an approximately central way, a conclusion that can be summarized by the existence of embeddings of tracially large, central matrix cones into the asymptotic sequence algebra.
Section \ref{MainProofSect} uses the conclusion of Section \ref{DivCentralSec} and $m$-almost comparison (and, again, locally finite nuclear dimension) to prove $\jsZ$-stability, using a characterization in terms of central dimension drop embeddings into the asymptotic sequence algebra modulo the annihilator.
Finally, it is shown in Section \ref{ASHSec} how the main result here combines with a main result of \cite{TT:ranks} to characterize slow dimension growth in approximately subhomogeneous $C^*$-algebras.

\subsection*{Acknowledgments}
The author would like to thank Joachim Cuntz, George Elliott, Dominic Enders, Ilijas Farah, Bhishan Jacelon, Leonel Robert, Luis Santiago, J\'{a}n \v{S}pakula, Hannes Thiel, Andrew Toms, Stuart White, and Wilhelm Winter for discussions promoting the research in this article.

\subsection{Notation}
The following function will appear many times in the sequel.
For $0 \leq \nu < \eta$, we define $g_{\nu,\eta} \in C_0((0,\infty])$ to be the function that is $0$ on $[0,\nu]$, $1$ on $[\eta,\infty]$ and linear on $[\nu,\eta]$.

\section{Algebraic simplicity and traces}
\alabel{AlgSimpleSec}

We begin by establishing a class of not-necessarily unital $C^*$-algebras which are amenable to many of the techniques used to study unital $C^*$-algebras.
The simple $C^*$-algebras in this class are algebraically simple -- which at first seems to be a very restrictive condition, and in particular, excludes $A \tens \K$ whenever $A$ is stably finite \ccite{Theorem 2.4}{Cuntz:On}.
However, we show in Corollary \ref{ExistsAlgSimple} that if $A$ is simple and $\sigma$-unital, then $A$ is stably isomorphic to an algebraically simple $C^*$-algebra (in this class).
We will make great use of Pedersen's minimal dense ideal, first studied in \cite{Pedersen:MeasureI}, which we will call the Pedersen ideal and denote (for a $C^*$-algebra $A$) by $\Ped(A)$.
A good reference on the Pedersen ideal is \ccite{Section 5.6}{Pedersen:CstarBook}.

\begin{thm}
\alabel{PedACharacterization}
Let $A$ be a $C^*$-algebra.
The following are equivalent:
\begin{enumerate}
\item \alabel{PedACharacterization-Ped} $A = \Ped(A)$ and $A$ is $\sigma$-unital;
\item \alabel{PedACharacterization-StrPos1} There exists a strictly positive element in $\Ped(A)$;
\item \alabel{PedACharacterization-StrPos2} There exists a strictly positive element for $A$ in $\Ped(A \tens \K) \cap A$.
\end{enumerate}
\end{thm}

\begin{proof}
\ref{PedACharacterization-Ped} $\Rightarrow$ \ref{PedACharacterization-StrPos1} $\Rightarrow$ \ref{PedACharacterization-StrPos2} are clear.

\ref{PedACharacterization-StrPos1} $\Rightarrow$ \ref{PedACharacterization-Ped} holds by \ccite{Proposition 5.6.2}{Pedersen:CstarBook}, which says that whenever $a \in \Ped(A)$, it follows that $\her{a} \subseteq \Ped(A)$.

\ref{PedACharacterization-StrPos2} $\Rightarrow$ \ref{PedACharacterization-StrPos1}: Let $e \in \Ped(A \tens \K) \cap A$ be strictly positive (for $A$).
Since $e$ is full in $A \tens \K$, it follows by \ccite{Proposition 3.1}{TT:ranks} that $\Prim(A \tens \K)$ is compact.
But $\Prim(A \tens \K) = \Prim(A)$, and then it follows by \ccite{Proposition 3.1}{TT:ranks} that $\Ped(A)$ contains a full element, $a$.

$a^3$ is also full, and so, $\Ped(A \tens \K)$ is the algebraic ideal of $A \tens \K$ generated by $a^3$.
By \ccite{Proposition 5.6.2}{Pedersen:CstarBook}, $e^{1/3} \in \Ped(A \tens \K)$.
Therefore, there exist $x_1,\dots,x_k,y_1,\dots,y_k \in A \tens \K$ such that
\[ e^{1/3} = \sum_{i=1}^k x_ia^3y_i, \]
and so
\[ e = \sum_{i=1}^k \left(e^{1/3}x_ia\right)\, a\, \left(ay_ie^{1/3}\right), \]
where $e^{1/3}x_ia, ay_ie^{1/3} \in A$.
Thus, $e$ is contained in the algebraic ideal of $A$ generated by $a$, which is clearly $\Ped(A)$.
\end{proof}

\begin{cor}
\alabel{ExistsAlgSimple}
Let $A$ be a $\sigma$-unital simple $C^*$-algebra.
Then there exists a nonzero hereditary subalgebra $B$ of $A$ such that $B$ is algebraically simple.
In particular, $A$ is stably isomorphic to $B$.
\end{cor}

\begin{proof}
We merely take $b \in \Ped(A)$ nonzero and set $B:=\her{b}$.
Brown's Theorem \cite{Brown:StableIsomorphism} shows that $A$ is stably isomorphic to $B$.
Moreover, since $\Ped(A) \subset \Ped(A \tens \K) = \Ped(B \tens \K)$, it follows from Theorem \ref{PedACharacterization} that $B$ is algebraically simple.
\end{proof}

\begin{defn}
We shall use $T(A)$ to denote the set of densely finite (a.k.a.\ densely defined) traces, as defined in \ccite{Definition 5.2.1}{Pedersen:CstarBook}.
Every $\tau \in T(A)$ is defined on all of $\Ped(A)$.
For $\tau \in T(A)$, we denote
\[ \|\tau\| := \sup \{\tau(x): x \in A_+, \|x\| \leq 1\} \in [0,\infty]; \]
and we set
\[ T^1(A) := \{\tau \in T(A): \|\tau\| = 1\}. \]

For $\tau \in T(A)$ and $a \in A_+$, we set
\[ d_\tau(a) := \lim_{n \to \infty} \tau(a^{1/n}). \]
\end{defn}

In the unital case, $T^1(A)$ consists of exactly the traces which take the value $1$ at the unit.
We record the following easily verified generalization to the nonunital setting.

\begin{prop}
Let $e \in A_+$ be strictly positive.
Then for any $\tau \in T(A)$,
\[ \|\tau\| = d_\tau(e). \]
\end{prop}

\begin{proof}
We may assume that $e$ is contractive, since rescaling $e$ does not change $d_\tau(e)$.

For $a \in A_+$,
\[ \tau(a) = \lim_{n \to \infty} \tau(e^{1/n}ae^{1/n}) \leq \lim_{n \to \infty} \|a\|\tau(e^{2/n}) = \|a\|d_\tau(e); \]
and therefore, $\|\tau\| \leq d_\tau(e)$.
On the other hand, $\tau(e^{1/n}) \leq \|\tau\|$ for all $n$, hence $\|\tau\| \geq d_\tau(e)$. 
\end{proof}

\begin{prop}
\alabel{PedAProperties}
Let $A$ satisfy the equivalent conditions of Theorem \ref{PedACharacterization}.
Then
\begin{enumerate}
\item \alabel{PedAProperties-Bdd} Every densely finite trace on $A$ is bounded;
\item \alabel{PedAProperties-BddBelow}If $a \in A_+$ is full then
\[ \inf_{\tau \in T^1(A)} \tau(a) > 0. \]
\end{enumerate}
\end{prop}

\begin{proof}
Let $e \in A_+$ be strictly positive, let $a \in A_+$ be full, and let $\tau \in T(A)$.
Since $e$ is in the Pedersen ideal, we have
\[ [e] \ll \infty[a] \]
in $\Cu(A)$, and therefore, there exists $M$ and $\e$ such that
\[ [e] \leq M[(a-\e)_+]. \]
It follows that
\begin{align*}
\|\tau\| &= d_\tau(e) \\
&\leq M d_\tau((a-\e)_+) \\
&\leq \frac M\e \tau(a). 
\end{align*}

\ref{PedAProperties-Bdd} now follows: since $a \in \Ped(A)$, we know that $\tau(a) < \infty$.

\ref{PedAProperties-BddBelow}  also follows since this shows that
\[ \inf_{\tau \in T^1(A)} \tau(a) \geq \frac\e M. \]
\end{proof}

We will most often use $T^1(A)$ to access the traces on $A$.
However, another base of $T(A)$ is at times better; it is
\[ T_{a \to 1}(A) := \{\tau \in T(A): \tau(a) = 1\}. \]

\begin{prop} \ccite{Proposition 3.4}{TT:ranks}
\alabel{TChoquet}
Suppose that $a \in \Ped(A)_+$ is full.
It follows that $T_{a \mapsto 1}(A)$ is:
\begin{enumerate}
\item a base for $T(A)$
\item compact, in the topology of pointwise converge on the Pedersen ideal of $A$;
\item a Choquet simplex.
\end{enumerate}
\end{prop}

%Verify this part with W/L
We will now consider some comparison and divisibility properties that appear in \ccite{Definitions 3.1 and 3.5}{Winter:pure}, in the unital case.
Some of these have already been phrased to (near) satisfaction in a nonunital setting: namely, $m$-comparison, as defined in \ccite{Definition 2.8}{OrtegaPereraRordam}, and $m$-almost divisible Cuntz semigroup, as defined in \ccite{Definition 3.5 (i)}{Winter:pure}.
One finds that the other properties generalize reasonably if we insist that we work with algebraically simple $C^*$-algebras (which do, indeed, generalize unital simple algebras).
However, one important modification has been made to the definitions of (tracial) $m$-almost-divisibility, namely that a nonzero requirement has been added.
It seems, given the way that these definitions are used, that this altered definition is what was actually intended in \cite{Winter:pure}.
The modification is, additionally, rather mild; see Remark \ref{DivNonzeroRmk} below.
The author would like to thank Leonel Robert for pointing out that this modification had been overlooked in an earlier version of this paper.

We restrict to exact $C^*$-algebras to avoid considering quasitraces (using the results of Haagerup \cite{Haagerup:quasitraces} and Blanchard-Kirchberg \cite{BlanchardKirchberg:pi} that all quasitraces on these algebras are traces; cf.\ also \cite{BrownWinter:quasitraces}).

\begin{defn}
\alabel{CompDivDefn}
Let $A$ be an algebraically simple, exact $C^*$-algebra and let $m \in \N$, $\dl > 0$
\begin{enumerate}
\item
We say that $A$ (or $W(A)$) has \textbf{$m$-comparison} if, given $[x],[y_0],\dots,[y_m] \in W(A)$, if there exists $k \in \mathbb{N}$ such that
\[ (k+1)[x] \leq k[y_i] \]
for $i=0,\dots,m$ then it follows that
\[ [x] \leq [y_0] + \cdots + [y_m]. \]
\item
We say that $A$ has \textbf{strong tracial $m$-comparison} if, for any nonzero contractions $a,b \in M_\infty(A)_+$, if
\[ d_\tau(a) < \frac1{m+1} \tau(b) \]
for all $\tau \in T^1(A)$ then
\[ [a] \leq [b] \]
in $W(A)$.
\item
We say that $A$ (or $W(A)$) is \textbf{$m$-almost divisible} if, for any $[a] \in W(A) \setminus \{0\}$ and any $k \in \mathbb{N}$, there exists $[x] \in W(A) \setminus \{0\}$ such that
\[ k[x] \leq [a] \leq (k+1)(m+1)[x]. \]
\item
We say that $A$ is \textbf{tracially $m$-almost divisible} if for any positive contraction $b \in M_\infty(A)_+, k \in \N, \e > 0$, there exists a nonzero c.p.c.\ order zero map $\psi:M_k \to \her{b}$ such that
\[ \tau(\psi(1_k)) \geq \dl\tau(b) - \e \]
for all $\tau \in T^1(A)$.
\end{enumerate}
\end{defn}

\begin{remark}
\alabel{DivNonzeroRmk}
If $T(A) \neq \emptyset$ then clearly the definitions of (tracial) $m$-almost-divisibility are unchanged if we don't ask that $[x]$ ($\psi$) is nonzero (since evaluating on a trace shows that they are automatically nonzero).
On the other hand, if $T(A) = \emptyset$ and $A$ is nonelementary, then Glimm's Halving Lemma \ccite{Lemma 6.7.1}{Pedersen:CstarBook} shows that $A$ is tracially $0$-almost divisible.
Likewise, if $T(A) = \emptyset$ and $A$ has $m$-comparison then, again using Glimm's Halving Lemma, one obtains that $A$ is $m$-almost divisible.
\end{remark}

Let us establish some relationships between these properties.

\begin{prop}
\alabel{CompDivRelationships}
Let $A$ be an algebraically simple, exact $C^*$-algebra.
\begin{enumerate}
\item \alabel{CDR-CompEquiv} The following are equivalent.
\begin{enumerate}
\item[(a)] $A$ has $m$-comparison;
\item[(b)] Given $[x],[y_0],\dots,[y_m] \in W(A) \setminus \{0\}$, if
\[ d_\tau(x) < d_\tau(y_i) \]
for all $\tau \in T^1(A)$ and for $i=0,\dots,m$ then it follows that
\[ [x] \leq [y_0] + \cdots + [y_m]; \quad\text{and}\]
\item[(c)] Given $[x],[y_0],\dots,[y_m] \in W(A) \setminus \{0\}$, if there exists $k \in \mathbb{N}$ such that
\[ d_\tau(x) < \tau(y_i) \]
for all $\tau \in T^1(A)$ and for $i=0,\dots,m$, then it follows that
\[ [x] \leq [y_0] + \cdots + [y_m]. \]
\end{enumerate}
\item \alabel{CDR-SCompEquiv} $A$ has strong tracial $m$-comparison if and only if, for any nonzero contractions $a,b \in M_\infty(A)_+$, if
\[ d_\tau(a) < \frac{1-\e}{m+1} \tau(b) \]
for all $\tau \in T^1(A)$ then
\[ [a] \leq [b] \]
in $W(A)$.
\item \alabel{CDR-CuDiv} If $A$ has stable rank one or finite radius of comparison then $W(A)$ is $m$-almost divisible if $\Cu(A)$ is $m$-almost divisible.
\item \alabel{CDR-TrDivEquiv} $A$ is tracially $m$-almost divisible if and only if for any $a \in M_\infty(A)_+, k \in \N, \e > 0$, there exists a nonzero c.p.c.\ order zero maps $\psi:M_k \to \her{a}$ such that
\[ \tau(\psi(1_k)) \geq \frac{1-\e}{m+1}\tau(a) \]
for all $\tau \in T(A)$.
\item \alabel{CDR-DivTrDiv} If $A$ is $m$-almost divisible then $A$ is tracially $m$-almost divisible.
\item \alabel{CDR-CompTrDivTrComp} For any $m$ there exists $\tilde m$ (independent of $A$) such that, if $A$ has $m$-comparison and is tracially $m$-almost divisible then $A$ has strong tracial $m$-comparison.
\end{enumerate}
\end{prop}

\begin{proof}
The arguments for \ref{CDR-CompEquiv} (b) $\Leftrightarrow$ (c), \ref{CDR-SCompEquiv}, \ref{CDR-DivTrDiv}, and \ref{CDR-CompTrDivTrComp} are exactly the same as the proofs of Propositions 3.3, 3.4, 3.8 and 3.9 in \cite{Winter:pure} respectively, as long as we use $T_{e \mapsto 1}(A)$ in place of $QT(A)$ (particularly for the use of Dini's Theorem), use the result of \ref{CDR-TrDivEquiv} in the proof of \ref{CDR-DivTrDiv}, and use the results of \ref{CDR-CompEquiv}, \ref{CDR-SCompEquiv}, and \ref{CDR-TrDivEquiv} in the proof of \ref{CDR-CompTrDivTrComp}.
Since the relevant inequalities are linear in the traces $\tau$ involved, they hold for all $\tau \in T_{e \to 1}(A)$ iff they hold for all $\tau \in T^1(A)$, iff they hold for all $\tau \in T(A)$.

\ref{CDR-CompEquiv} (a) $\Leftrightarrow$ (b): This follows by \ccite{Proposition 2.1}{OrtegaPereraRordam} (the essential ingredient there being \ccite{Lemma 4.1}{GoodearlHandelman:Ranks}).

\ref{CDR-CuDiv} follows from the fact that $W(A)$ is hereditary in $\Cu(A)$ under the stated hypotheses, by \ccite{Theorem 4.4.1}{BRTTW} and \ccite{Theorem 3}{CowardElliottIvanescu} respectively.

\ref{CDR-TrDivEquiv} follows from the observation that, fixing $a \in M_\infty(A)_+$, there exist $r,R \in (0,\infty)$ such that
\[ \tau(a) \in (r, R) \]
for all $\tau \in T^1(A)$ (which in turn follows from Proposition \ref{PedAProperties}).
\end{proof}

We now establish some technical results concerning traces.

\begin{lemma}
\alabel{TraceExtension}
Let $\tau \in T(A)$ be bounded.
Then $\tau$ extends uniquely to a bounded trace $\tilde\tau$ on $\mathcal{M}(A)$ satisfying $\|\tilde\tau\| = \|\tau\|$; for $b \in \mathcal{M}(A)$,
\begin{align*}
\tilde\tau(b) &= \sup \{\tau(aba): a \in A, \|a\| \leq 1\}; \\
&= \sup_\lambda \tau(e_\lambda be_\lambda)
\end{align*}
for any approximate identity $(e_\lambda)$.
\end{lemma}

\begin{proof}
Let us first take $(e_\lambda)$ to be an increasing approximate unit (which exists by \ccite{Theorem 1.4.2}{Pedersen:CstarBook}), and show that
\[ \tilde\tau(b) := \sup_\lambda \tau(e_\lambda b e_\lambda) \]
defines a bounded trace on $\mathcal{M}(A)$.

Note that if $\lambda < \lambda'$ then $e_\lambda^2 \leq e_{\lambda'}^2$ and so
\[ \tau\left(e_\lambda b e_\lambda\right) = \tau\left(b^{1/2}e_\lambda^2b^{1/2}\right) \leq \tau\left(b^{1/2}e_{\lambda'}^2b^{1/2}\right), \]
and therefore,
\[ \tilde\tau(b) = \lim_\lambda \tau(e_\lambda b e_\lambda). \]
From this, it is clear that $\tilde\tau$ is additive and positive.
It is also evident that $\tilde\tau|_A = \tau$.

To see that it satisfies the trace identity, note that for any $x \in \mathcal{M}(A)$ and any $\lambda$, we have
\[ \tau(e_\lambda x^*xe_\lambda) = \tau(x^*e_\lambda^2x) = \tilde\tau(x^*e_\lambda^2x) \leq \tilde\tau(x^*x), \]
and therefore, $\tilde\tau(x^*x) \leq \tilde\tau(xx^*)$.
The trace identity follows by symmetry.

Let us now show that, for any (not necessarily increasing) approximate identity $(e_\lambda)$, and any $b \in \mathcal{M}(A)$, we have
\[ \sup \{\tau(aba): a \in A_+, \|a\| \leq 1\} = \sup_\lambda \tau(e_\lambda be_\lambda). \]

The inequality $\geq$ is automatic.
Conversely, for any contractive $a \in A_+$ and any $\e > 0$, we can find $\lambda$ such that $e_\lambda a \approx_\e a$; it follows that
\begin{align*}
\tau(aba) &\approx_{2\e\|\tau\|\|b\|} \tau(a e_\lambda b e_\lambda a) \\
&= \tau\left(b^{1/2} e_\lambda a^2 e_\lambda b^{1/2}\right) \\
&\leq \tau\left(b^{1/2} e_\lambda^2 b^{1/2}\right) \\
&= \tau(e_\lambda b e_\lambda).
\end{align*}
Since $\e$ is arbitrary, this establishes the inequality $\leq$.

This shows that the formulae for $\tilde\tau$ are equivalent (and by the above, $\tilde\tau$ is a tracial functional).
%In addition to verifying that the different formulae for $\tilde\tau$ are equivalent, we can now conclude that
%\[ b \mapsto \sup \{\tau(aba): a \in A\} \]
%satisfies the trace identity.

Finally, we show uniqueness.
Suppose that $\tau_0$ is another trace on $\mathcal{M}(A)$ such that $\tau_0|_A = \tau$ and $\|\tau_0\| = \|\tau\| = \|\tilde\tau\|$.
Notice that we clearly have $\tilde\tau \leq \tau_0$, and in particular, $\tau_0 -\tilde\tau$ is itself a trace.
Moreover,
\[ \|\tau_0 - \tilde\tau\| = \tau_0(1) - \tilde\tau(1) = \|\tau_0\| - \|\tilde\tau\| = 0, \]
and therefore, $\tau_0 = \tilde\tau$.
\end{proof}

\begin{prop}
\alabel{Trace1}
Let $a \in A_+$ be a positive contraction and suppose $\tau \in T^1(A)$ satisfies $\tau(a) = 1$.
Then
\[ \tau(a^{1/2}xa^{1/2}) = \tau(x) \]
for all $x \in A_+$ and
\[ \tau(f(a)) = f(1) \]
for all $f \in C_0((0,1])_+$.
\end{prop}

\begin{proof}
Define $\tilde\tau \in T(\mathcal{M}(A))$ by Lemma \ref{TraceExtension}.
Then
\begin{align*}
\tau(x) - \tau(a^{1/2}xa^{1/2}) &= \tau((1-a)x) \\
&= \tau(x^{1/2}(1-a)x^{1/2}) \\
&\leq \tilde\tau(1-a) \\
&= 1 - \tau(a) = 0;
\end{align*}
on the other hand, $x^{1/2}ax^{1/2} \leq x$ provides the other inequality, so that
\[ \tau(x) = \tau(a^{1/2}xa^{1/2}). \]

Elementary computation shows that the second statement follows from the first, in the case that $f$ is a polynomial.
For arbitrary $f \in C_0((0,1])$, $\tau(f(a)) = f(1)$ follows by approximating by $f$ polynomials.
\end{proof}

\section{Nuclear dimension and order zero maps}
\alabel{DimNucSec}

The definition of nuclear dimension rests on the notion of order zero completely positive contractive maps.
They were first defined in \ccite{Definition 2.1 (b)}{Winter:CovDim1} for finite dimensional domains (which is all that is used in the present article).
The definition in \ccite{Definition 2.1 (b)}{Winter:CovDim1} is slightly different from, though equivalent (in the case of finite dimensional domains) to, the more elegant definition we state here.

\begin{defn}\ccite{Definition 1.3}{WinterZacharias:Order0}
Let $A,B$ be $C^*$-algebras and let $\phi:A \to B$.
We say that $\phi$ is a \textbf{completely positive contractive (c.p.c.) order zero} map if it is completely positive, contractive, and preserves orthogonality in the sense that if $a,b \in A_+$ satisfy $ab=0$ then $\phi(a)\phi(b)=0$.
\end{defn}

\begin{defn}\ccite{Definition 2.1}{WinterZacharias:NucDim}
Let $A$ be a $C^*$-algebra.
We say that $A$ has \textbf{nuclear dimension }$n$ if $n$ is the least number for which, given $\e > 0$ and a finite subset $\mathcal{F} \subset A$, there exists a finite dimensional $C^*$-algebra $F$ and maps
\[ A \labelledrightarrow{\psi} F \labelledrightarrow{\phi} A \]
such that:
\begin{enumerate}
\item for all $a \in \mathcal{F}$, $\|\phi\psi(a)-a\| < \e$;
\item $\psi$ is c.p.c.; and
\item $F$ decomposes as
\[ F=F_0 \dsum \cdots \dsum F_n \]
such that $\phi|_{F_i}$ is a c.p.c.\ order zero map for each $i$.
\end{enumerate}
We will call such
\begin{equation}
\alabel{DecompApproximation}
 A \labelledrightarrow{\psi} F \labelledrightarrow{\phi} A
\end{equation}
an \textbf{$(n+1)$-colourable c.p.\ approximation} of $\mathcal{F}$ to within $\e$.
\end{defn}

%(Note that in the literature, \eqref{DecompApproximation} has been called $n$-decomposable rather than $(n+1)$-colourable; I hope that my modified notation is less confusing.)

The following is an extremely useful result, completely revealing the structure of order zero maps from finite dimensional $C^*$-algebras.
(It also holds for order zero maps from arbitrary $C^*$-algebras \ccite{Corollary 3.1}{WinterZacharias:Order0}, though this generalization won't be used here.)

\begin{prop} \ccite{Proposition 3.2}{Winter:CovDim1}
\alabel{OrderZeroStructure}
Let $A,F$ be $C^*$-algebras with $F$ finite dimensional, and let $\phi:F \to A$ be a c.p.c.\ order zero map.
Then there exists a $*$-homomorphism $\hat\phi:C_0((0,1],F) \to A$ such that
\[ \phi(x) = \hat\phi(\id_{(0,1]} \tens x) \]
for all $x \in F$.
\end{prop}

We now record a fundamental fact concerning colourable c.p.\ approximations.

\begin{lemma}
\alabel{GoodDnApproximations}
(cf.\ \ccite{Proposition 4.2}{Winter:pure}
Fix $n$.
Given a finite subset $\mathcal{F} \subset A$ and a tolerance $\e > 0$, there exists $\dl > 0$ and a finite subset $\mathcal{G} \subset A$ such that, if $e \in A_+$ is a positive contraction satisfies
\[ ey \approx_\dl y \]
for all $y \in \mathcal{F}$, and
\[ (F = F^{(0)} \dsum \cdots \dsum F^{(n)}, \psi, \phi) \]
is a $(n+1)$-colourable c.p.\ approximation of $\mathcal{G} \cup \{e\}$ to within $\dl$ then, for all $a \in \mathcal{F}$ and $i=0,\dots,n$,
\begin{enumerate}
\item $\|[\phi\psi(a),\phi^{(i)}\psi^{(i)}(e)]\| < \e$; 
\item $a\phi^{(i)}\psi^{(i)}(e) \approx_\e \phi^{(i)}\psi^{(i)}(a)$; and
\end{enumerate}
\end{lemma}

\begin{proof}
This follows from the first part of the proof of \ccite{Proposition 4.3}{WinterZacharias:NucDim} and some bookkeeping, although a typo should be pointed out that slightly obscures this: Equation (5) of \cite{WinterZacharias:NucDim} should read
\[ \|\phi^{(i)}_\lambda \tilde\psi^{(i)}_\lambda(a) - \phi^{(i)}_\lambda \tilde\psi^{(i)}_\lambda(e_\lambda) \phi_\lambda \tilde\psi(a)\| \labelledrightarrow{\lambda \to \infty} 0. \]
\end{proof}

\section{The asymptotic sequence algebra}
\alabel{AsympSeqSec}

The asymptotic sequence algebra features prominently in the main arguments of this article (even more so than in \cite{Winter:pure}).
It is used to concisely convey approximate notions; it is not essential, in the sense that statements and arguments could all be done without it, in the language of approximation (see Proposition \ref{SequenceTraces} for an example of how one translates between the asymptotic sequence algebra and approximation).

\begin{defn}
Let $A$ be a $C^*$-algebra.
The \textbf{asymptotic sequence algebra} of $A$ is defined as
\[ \textstyle{A_\infty := \prod_{n=1}^\infty A \big/ \Big\{(a_n) \in \prod_{n=1}^\infty A: \lim_{n \to \infty} \|a_n\| = 0\Big\}.} \]
We shall use $\pi_\infty:\prod_{n=1}^\infty A \to A_\infty$ to denote the quotient map and $\iota:A \to A_\infty$ to denote the canonical inclusion given by $\iota(a) := \pi_\infty(a,a,\dots)$.
\end{defn}

\begin{defn}
Let $A$ be a $C^*$-algebra.
We set denote by $T^1_\infty(A)$ the set of all $\tau \in T(A_\infty)$ for which there exists a sequence $(\tau_n) \subseteq T^1(A)$ and an ultrafilter $\omega$ such that
\begin{equation}
\alabel{Tinfty-Defn}
 \tau(\pi_\infty(a_n)) = \lim_{n \to \omega} \tau_n(a_n)
\end{equation}
for all $\pi_\infty(a_n) \in A_\infty$.
\end{defn}

\begin{remark*}
Notice that for any sequence $(\tau_n) \subseteq T^1(A)$ and any ultrafilter $\omega$, the right-hand side of \eqref{Tinfty-Defn} defines a trace of norm one in $T^1_\infty(A)$.
In particular, $T^1_\infty(A) \subseteq T^1(A)$.
\end{remark*}

\begin{prop}
\alabel{SequenceTraces}
Let $(a_n) \subset A$ be a bounded sequence of self-adjoint elements and set $a := \pi_\infty(a_n) \in A_\infty$.
Let $r \in \mathbb{R}$.
Then:
\begin{enumerate}
\item \alabel{SequenceTraces-inf} $\tau(a) > r$ for all $\tau \in T^1_\infty(A)$ if and only if
\[ \liminf_{n \to \infty} \inf_{\tau \in T^1(A)} \tau(a_n) > r; \]
\item \alabel{SequenceTraces-sup} $\tau(a) < r$ for all $\tau \in T^1_\infty(A)$ if and only if
\[ \limsup_{n \to \infty} \sup_{\tau \in T^1(A)} \tau(a_n) < r. \]
\end{enumerate}
\end{prop}

\begin{proof}
\ref{SequenceTraces-inf}: The reverse implication is quite evident.
Conversely, suppose that
\[ \liminf_{n \to \infty} \inf_{\tau \in T^1(A)} \tau(a_n) \leq r. \]
This means that there exists an unbounded infinite sequence $(n_k)$ of indices, together with $\tau_{n_k} \in T^1(A)$ for each $k$ such that
\[ \lim_{k \to \infty} \tau_{n_k}(a_{n_k}) \leq r. \]
Pick $\tau_n$ arbitrarily for $n \not\in \{n_k\}$, and let $\omega$ be an ultrafilter containing the set $\{n_k\}$.
Then set $\tau := \lim_{n \to \omega} \tau_n(\cdot) \in T^1_\infty(A)$, and we find that
\[ \tau(a_n) = \lim_k \tau_{n_k}(a_{n_k}) \leq r, \]
as required.

\ref{SequenceTraces-sup} follows from \ref{SequenceTraces-inf} by replacing $a_n$ by $-a_n$.
\end{proof}

\subsection{The diagonal sequence argument}
\alabel{DiagArgSect}
The diagonal sequence argument will be used very frequently to allow us to prove the existence of elements in $A_\infty$ satisfying some conditions (exactly), by only verifying that the conditions hold approximately.
In order to use the diagonal sequence argument, the rules are as follows:
\begin{enumerate}
\item[(a)]
The conditions need to take the form $f(a) = 0$, where $f:A_\infty \to [0,\infty]$ is a function of the form
\begin{equation}
\alabel{DiagArgCondition}
 f(\pi_\infty(a_n)) = \limsup_{n \to \infty} f_n(a_n),
\end{equation}
for some functions $f_n$; we say that $a$ approximately satisfies the condition (up to a tolerance $\e > 0$) if $f(a) < \e$.
(At times, we seek not one but finitely or even countably many elements satisfying joint conditions, and therefore ask for multivariable versions of \eqref{DiagArgCondition})
\item[(b)] 
We may use at most countably many conditions; abusing terminology which is anyhow meant to be informal, we shall use the name ``admissable condition'' to refer to a countable collection of such conditions (or a statement that is equivalent to a countable collection of such conditions).
Approximately satisfying a countable family of conditions means approximately satisfying (up to $\e$) a (specified) finite set of them. 
\item[(c)]
We must include a condition that implies that $\|a\| \leq R$ for some (fixed) $R$.
\end{enumerate}

Here are important examples of admissable conditions.
\begin{enumerate}
\item \alabel{DA-comm} Given a separable subspace $X \subset A_\infty$, we have the condition on $a \in A_\infty$ that
\begin{equation}
\alabel{DiagArg-Commuting}
[a,X] = 0;
\end{equation}
approximately satisfying this condition means, for a finite subset $\mathcal{F} \subset X$ and a tolerance $\e > 0$,
\[ \|[a,x]\| < \e \]
for all $x \in \mathcal{F}$.
Certainly, we find that \eqref{DiagArg-Commuting} is equivalent to the countable collection of conditions $f^{(k)}(a) = \|[x_k,a]\|$, where $(x_k) \subset X$ is a dense sequence.
It is not hard to see that these $f^{(k)}$'s have the form \eqref{DiagArgCondition}.

\item \alabel{DA-unit1} Given a separable subspace $X \subset A_\infty$, we have the condition on $a \in A_\infty$ that
\[ ax=x, \quad \forall\, x \in X; \]
approximately satisfying this condition means, for a finite subset $\mathcal{F} \subset X$ and a tolerance $\e > 0$,
\[ \|ax-x\| < \e \]
for all $x \in \mathcal{F}$.
Again, a sequence of conditions of the form \eqref{DiagArgCondition} can be formed using a dense sequence from $X$.

\item \alabel{DA-trace} Given a separable set $X \subset (A_\infty)_+$ and $\gamma > 0$, we have the condition on $a \in (A_\infty)_+$ that
\[ \tau(a^{1/2}xa^{1/2}) \geq \gamma \tau(x) \quad \forall\,x \in X, \tau \in T^1_\infty(A); \]
approximately satisfying this condition means, for a finite subset $\mathcal{F} \subset X$ and a tolerance $\e > 0$,
\[ \tau(a^{1/2}xa^{1/2}) \geq \gamma \tau(x) - \e \quad \forall\,x \in \mathcal{F}, \tau \in T^1_\infty(A); \]
or
\[ \tau(a^{1/2}xa^{1/2}) \geq (\gamma-\e) \tau(x) \quad \forall\,x \in \mathcal{F}, \tau \in T^1_\infty(A). \]

\item \alabel{DA-unit2} Given $e \in (A_\infty)_+$, we have the condition on $a \in A_\infty$ that
\[ ea = a; \]
approximately satisfying this condition means, for a tolerance $\e > 0$,
\[ \|ea-a\| < \e. \]

\item \alabel{DA-ordzero} A linear map $M_k \to A_\infty$ may be encoded by its behaviour on a basis of $M_k$, and therefore we may seek linear maps with certain conditions.
Such a condition is that the map is c.p.c.\ order zero. 
Although we could form an approximate version of this condition, we don't as it is not needed in the sequel.

\item \alabel{DA-pos} A special case of \ref{DA-ordzero} is, for a single element $a$, the condition $a \geq 0$.
\end{enumerate}

\begin{remark*}
We warn that we must be careful not to involve conditions that can't be encoded by a countable set of functions of the form \eqref{DiagArgCondition}, even if we verify them exactly instead of approximately.
The condition $a \in \her{b}$ is an example of a condition which is not allowed, as the following example shows (we generally employ condition \ref{DA-unit2} whenever we wish we could use this condition). 

Let $A$ be a separable $C^*$-algebra which does not have an approximate identity consisting of projections, and let $b \in A_+$ be strictly positive.
We shall show that, even though there exists $a \in \overline{bA_\infty b}_+$ satisfying $ab=b$ approximately (this is an instance of \ref{DA-unit1}), there does not exist $a \in \overline{bA_\infty b}_+$ satisfying $ab=b$ exactly.
Certainly, using an approximate identity, we get $a \in A_+ \subset \overline{bA_\infty b}_+$ satisfying $ab \approx b$.
But, if $a \in \overline{bA_\infty b}_+$ satisfies $ab=b$ then it follows that $aa=a$.
This means that $a$ is a projection, and can therefore be lifted to a projection in $\prod_{n=1}^\infty A$, and since $ab=b$, its lift is an approximate identity, contradicting our hypothesis.
\end{remark*}

Here is a proof of the diagonal sequence argument.
We only state it in the one-variable case, but the $k$-variable (or countably-many variable) case follows from using $A^{\dsum k}$ (respectively $\bigdsum_{n=1}^\infty A$) in place of $A$.

\begin{prop}
Let $(f_n)$ be a sequence of functions of the form \eqref{DiagArgCondition}, and let $R > 0$.
Suppose that we can approximately satisfy these conditions by elements of norm at most $R$, i.e.\ for every $\e > 0$ and every $N$, there exists $a \in A_\infty$ such that $\|a\| \leq R$ and
\[ f_n(a) < \e\quad \forall\, n=1,\dots,N. \]
Then we can exactly satisfy these conditions, i.e.\ there exists $a \in A_\infty$ such that $\|a\| \leq R$ and
\[ f_n(a) = 0 \quad \forall\, n \in \mathbb{N}. \]
\end{prop}

\begin{proof}
By hypothesis, let $a_N := \pi_\infty\left((a^{(N)}_i)_{i=1}^\infty\right) \in A_\infty$ satisfy $\|a_N\| \leq R$.
\begin{equation}
\alabel{DiagArgArgument-Proof} f_n(a_N) < 1/N
\end{equation}
for $n=1,\dots,N$.
Express
\[ f_n(\pi_\infty(x_i)) = \limsup_i f^{(n)}_i(x_i) \]
for some functions $f^{(n)}_i$.
Then \eqref{DiagArgArgument-Proof} means that there exists $i_N$ such that
\[ \|a^{(N)}_i\| < R+1/N \]
and
\[ f^{(n)}_i(a^{(N)}_i) < 1/N \]
for all $n=1,\dots,N$ and all $i \geq i_N$.
We may assume that $(i_N)$ is an increasing sequence.
Define
\[ a_i := a^{(N)}_i \]
for $i=i_N,\dots,i_{N+1}-1$, and set
\[ a := \pi_\infty(a_1,a_2,\dots) \in A_\infty. \]
One readily verifies that $\|a\| \leq R$ and $f_n(a) = 0$ for all $n$.
\end{proof}

\subsection{Almost comparison, almost divisibility, and the asymptotic sequence algebra}

Here, we reformulate strong tracial comparison and tracial almost-divisibility in terms of the asymptotic sequence algebra.

\begin{defn}
We will say that $a \in A_\infty$ is \textbf{strongly nonzero} if it lifts to a sequence $(a_i) \in \prod_{n=1}^\infty$ such that $\inf_n \|a_n\| > 0$.
\end{defn}

\begin{lemma}
\alabel{SequenceAlgComparison}
Let $A$ be an algebraically simple, exact $C^*$-algebra.
$A$ has strong tracial $m$-comparison if and only if for any strongly nonzero positive contractions $a,b \in M_\infty(A_\infty)_+$, if for all $\e > 0$,
\[ \inf_{\tau \in T^1_\infty(A)} \left(\frac1{m+1} \tau(b) - \tau(g_{0,\e}(a))\right) > 0, \]
then $[a] \leq [b]$ in $W(A_\infty)$.
\end{lemma}

\begin{proof}
($\Rightarrow$): Suppose that $A$ has strong tracial $m$-comparison and that $a,b \in M_n(A_\infty)_+$ are strongly nonzero positive contractions which satisfy, for all $\e > 0$,
\[ r_\e := \inf_{\tau \in T^1_\infty(A)} \left(\frac1{m+1} \tau(b) - \tau(g_{0,\e}(a))\right) > 0. \]
Let $a = \pi_\infty(a_i)$ and $b = \pi_\infty(b_i)$ for some sequences $(a_i),(b_i) \subset M_n(A)_+$ of positive contractions, such that
\[ s := \inf_i \|b_i\| > 0. \]
Let $\e > 0$.
Pick some
\[ \eta < \min\{r_\e, s\}, \] %\inf_{\tau \in T^1_\infty(A)} \frac1{m+1} \tau(b) - \tau(g_{0,\e}(a))\right), \]
and we see that 
\[ \tau(g_{0,\e}(a)) \leq \frac1{m+1}\tau((b-\eta)_+)\quad \text{for all } \tau \in T^1_\infty(A). \]
By (the proof of) Proposition \ref{SequenceTraces}, for all $i$ sufficiently large, we then have
\[ d_\tau((a_i-\e)_+) \leq \tau(g_{0,\e}(a_i)) < \frac1{m+1} \tau((b_i-\eta)_+) \quad \text{for all } \tau \in T^1_\infty(A), \]
and since $\eta < s$, $(b_i-\eta)_+ \neq 0$.
By strong tracial $m$-comparison, $[(a_i-\e)_+] \leq [(b_i-\eta)_+]$ in $W(A)$.
\ccite{Proposition 2.4}{Rordam:UHFII} shows that there exists $x_i \in M_n(A)$ such that
\[ x_i^*x_i = (a_i-\e)_+ \quad\text{and}\quad x_ix_i^* \in \her{(b_i-\eta)_+}. \]
In particular, $\|x_i\| \leq 1$ and $g_{0,\eta}(b_i)$ acts as a unit on $x_ix_i^*$.

Now, set $x = \pi_\infty(x_i) \in A_\infty$, and note that 
\[ x^*x = (a-\e)_+ \quad\text{and}\quad g_{0,\eta}(b)xx^* = xx^*. \]
Thus, $[(a-\e)_+] \leq [b]$ in $W(A_\infty)$, and since $\e$ is arbitrary, $[a] \leq [b]$.

($\Leftarrow$): Given that $A$ satisfies the latter property in the statement of the lemma, let us show that $A$ has strong tracial $m$-comparison.
Thus, let us take nonzero positive contractions $a,b \in M_\infty(A)_+$ such that
\[ d_\tau(a) < \frac1{m+1} \tau(b) \]
for all $\tau \in T^1(A)$.
Let $\e > 0$, and let us show that
\[ \inf_{\tau \in T^1(A)} \left(\frac1{m+1}\tau(b) - \tau(g_{0,\e}(a))\right) > 0. \]
Certainly, if we fix a strictly positive element $e \in A_+$ then $\tau \mapsto \frac1{m+1}\tau(b) - \tau(g_{0,\e}(a))$ is a continuous function on the compact space $T_{e \mapsto 1}(A)$, and therefore, it has a minimum $r > 0$.
This means that
\[ \frac1{m+1}\tau(b) - \tau(g_{0,\e}(a)) > r\tau(e) \quad\text{for all }\tau \in T(A). \]
Moreover, by Proposition \ref{PedAProperties} \ref{PedAProperties-BddBelow}, there exists $r'> 0$ such that $\tau(e) > r'$ for all $\tau \in T^1(A)$.
Thus,
\[ \frac1{m+1}\tau(b) - \tau(g_{0,\e}(a)) > rr' \quad\text{for all } \tau \in T^1(A). \]
Since $\e$ is arbitrary, we have by hypothesis that $[a] \leq [b]$ in $W(A_\infty)$.
By virtue of the approximative nature of Cuntz comparison, this implies that $[a] \leq [b]$ in $W(A)$.
\end{proof}

\begin{lemma}
\alabel{SequenceAlgDiv}
Suppose that $A$ is an algebraically simple, exact $C^*$-algebra.
Then $A$ is tracially $m$-almost divisible if and only if, for any strongly nonzero positive contractions $d,e \in (A_\infty)_+$ such that $ed=d$, and for any $k \in \N$ there exists a c.p.c.\ order zero map
\[ \psi:M_k \to A_\infty \]
such that $e\psi(x)=\psi(x)$ for all $x \in M_k$, $\psi(1)$ is strongly nonzero, and
\[ \tau(\psi(1)) \geq \frac1{m+1}\tau(d) \quad \text{for all } \tau \in T^1_\infty(A). \]
\end{lemma}

\begin{proof}
The forward implication is a direct application of the diagonal sequence argument from Section \ref{DiagArgSect}.

Conversely, suppose that $A$ satisfies the latter property and that $b \in A_+$ is a positive contraction, $k \in \mathbb{N}$, and $\e > 0$ are given as in the definition of tracially $m$-almost divisible.
By hypothesis, we may find a c.p.c.\ order zero map $\psi:M_k \to A_\infty$ such that
\[ g_{0,\e/2}(b)\psi(x) = \psi(x) \]
for all $x \in M_k$ and
\[ \tau(\psi(1)) \geq \frac1{m+1}\tau((b-\e/2)_+ \quad \text{for all } \tau \in T^1_\infty(A). \]
In particular, $\tau(\psi(1)) > \frac1{m+1}\tau(b) - \e$ for all $\tau \in T^1_\infty(A)$.

By stability of c.p.c.\ order zero maps \ccite{Proposition 1.2.4}{Winter:CovDim2}, we may lift $\psi$ to a sequence of c.p.c.\ order zero maps $\psi_i:M_k \to A$ satisfying
\[ g_{0,\e/2}(b)\psi_i(x) = \psi_i(x) \quad \text{for all }x \in M_k,\]
and therefore, $\psi_i(M_k) \subseteq \her{b}$.
Then by (the proof of) Proposition \ref{SequenceTraces} \ref{SequenceTraces-inf}, for $i$ sufficiently large, we have $\tau(\psi_i(1)) > \frac1{m+1}\tau(b) - \e$ for all $\tau \in T^1(A)$, as required.
\end{proof}

\section{Characterizing $\jsZ$-stability}
\alabel{ZSec}

The Jiang-Su algebra $\jsZ$ was originally defined in \cite{JiangSu}; we will recall its characterization in Theorem \ref{JiangSuCharacterization}.
It is a strongly self-absorbing $C^*$-algebra in the sense of \ccite{Definition 1.3 (iv)}{TomsWinter:ssa}.

\begin{comment}
\begin{defn}(\ccite{Definition 1.3 (iv)}{TomsWinter:ssa})
A $C^*$-algebra $\mathcal{D}$ is \textbf{strongly self-absorbing} if $\mathcal{D} \neq \C$ is unital and there exists an isomorphism 
\[ \mathcal{D} \to \mathcal{D} \tens_{min} \mathcal{D} \]
which is approximately unitarily equivalent to the first-factor embedding 
\[ \mathcal{D} \to \mathcal{D} \tens 1_{\mathcal{D}} \subset \mathcal{D} \tens_{min} \mathcal{D}, \]
(by which we mean the map $d \mapsto d \tens 1_{\mathcal{D}}$).
\end{defn}
\end{comment}

The literature contains a few characterizations of $\mathcal{Z}$-stability involving suitable embeddings into asymptotic sequence algebras.
To this list, we shall add related characterizations, Propositions \ref{DStabilityCharacterization} and \ref{LimitDStabilityCharacterization}, which will be useful in the sequel for proving $\jsZ$-stability for classes of nonunital $C^*$-algebras.

For a $C^*$-algebra $A$, let us define
\[ A^\perp = \{b \in A_\infty: bA = Ab = 0\}. \]

\begin{prop}
\alabel{AnnQuotientUnital}
Let $A$ be a $\sigma$-unital $C^*$-algebra.
Then $A^\perp$ is an ideal of $A_\infty \cap A'$ and $(A_\infty \cap A')/A^\perp$ is unital.
\end{prop}

\begin{proof}
That $A^\perp$ is an ideal is an easy calculation.
The proof of \ccite{Proposition 1.9 (3)}{Kirchberg:CentralSequences} shows that $(A_\infty \cap A')/A^\perp$ is unital --- namely, if $(e_n) \subset A_+$ is an approximate unit of $A$ then its class in $(A_\infty \cap A')/A^\perp$ is the unit.
\end{proof}

\begin{prop} (cf.\ \ccite{Proposition 4.11}{Kirchberg:CentralSequences})
\alabel{DStabilityCharacterization}
Let $A$ be a separable $C^*$-algebra and let $\mathcal{D}$ be strongly self-absorbing.
Then $A$ is $\mathcal{D}$-stable if and only if there exists a unital $*$-homomorphism
\[ \mathcal{D} \to (A_\infty \cap A')/A^\perp. \]
\end{prop}

\begin{proof}
To show the forward implication, we shall describe a unital $*$-homomorphism
\[ \mathcal{D} \to ((A \tens \mathcal{D})_\infty \cap (A \tens \mathcal{D})')/(A \tens \mathcal{D})^\perp \]
for any separable $C^*$-algebra $A$.

Let $(e_n) \subset A_+$ be an approximate unit and let $\psi_n:\mathcal{D} \to \mathcal{D}$ be an approximately central sequence of unital $*$-homomorphisms.

Define a c.p.c.\ map 
\[ \phi_n := e_n \tens \psi_n(\cdot):\mathcal{D} \to A \tens \mathcal{D}. \]
For $a \in A$ and $d_1,d_2 \in \mathcal{D}$,
\[ \phi_n(d_1)(a \tens d_2) = (e_na) \tens (\psi(d_1)d_2). \]
It is clear from this that $\phi := \pi_\infty \circ (\phi_1,\phi_2,\dots):\mathcal{D} \to (A \otimes \mathcal{D})_\infty$ has image contained in the commutant of $A \otimes \mathcal{D}$.
Moreover, for $d_1,d_2,d_3 \in \mathcal{D}$ and $a \in A$,
\begin{align*}
 \phi_n(d_1)\phi_n(d_2)(a \tens d_3) - \phi_n(d_1d_2)(a \tens d_3) &= (e_n^2a) \tens (\psi_n(d_1)\psi_n(d_2)d_3) \\
&\qquad - (e_na) \tens (\psi_n(d_1d_2)d_3) \\
&= e_n(e_na-a) \tens \psi_n(d_1d_2)d_3 \\
&\to 0
\end{align*}
as $n \to \infty$.
Consequently, we see that $\phi$ in fact induces a $*$-homomorphism
\[ \mathcal{D} \to ((A \tens \mathcal{D})_\infty \cap (A \tens \mathcal{D})')/(A \tens \mathcal{D})^\perp; \]
it is straightforward (from the description of the unit in the proof of Proposition \ref{AnnQuotientUnital}) that this map is unital.

Conversely, given a unital $*$-homomorphism 
\[ \phi:\mathcal{D} \to (A_\infty \cap A')/A^\perp, \]
we may define a map $\psi:A \tens \mathcal{D} \to A_\infty$ by
\begin{equation}
 \psi(a \tens d) := ax 
\alabel{DStabilityCharacterization-PsiDef}
\end{equation}
where $x \in A_\infty \cap A'$ is a lift of $\phi(d)$.
Note that the definition of $A^\perp$ ensures that the right-hand side of \eqref{DStabilityCharacterization-PsiDef} is independent of the choice of $x$, i.e.\ that $\psi$ is well-defined.
Moreover, we see that since $\phi$ is unital, $\psi(a \tens 1) = a$ for any $a \in A$.
Thus, \ccite{Theorem 2.3}{TomsWinter:ssa} shows that $A$ is $\mathcal{D}$-stable (the hypothesis in \ccite{Theorem 2.3}{TomsWinter:ssa} that $\mathcal{D}$ is $K_1$-injective is automatically satisfied by \ccite{Remark 3.3}{Winter:Zinitial}).
\end{proof}

Related to the previous characterization of $\mathcal{D}$-stability is the following characterization, in the case that $\mathcal{D}$ is an inductive limit.
(We have in mind $\jsZ$ as a limit of dimension drop algebras; see Theorem \ref{JiangSuCharacterization} and the following comments below.)

\begin{prop}
\alabel{LimitDStabilityCharacterization}
Suppose that $\mathcal{D}$ is a strongly self-absorbing $C^*$-algebra that can be expressed as 
\[\textstyle{ \mathcal{D} = \overline{\bigcup_{i=1}^\infty \mathcal{D}_i}}, \]
where $\mathcal{D}_i$ is a unital nuclear subalgebra for each $i$, and the sequence $(\mathcal{D}_i)$ is increasing.
For a separable $C^*$-algebra $A$, $A$ is $\mathcal{D}$-stable if and only if, for every $i$ there exists a unital $*$-homomorphism
\[ \mathcal{D}_i \to (A_\infty \cap A')/A^\perp. \]
\end{prop}

\begin{proof}
Appealing to Proposition \ref{DStabilityCharacterization}, we will show that the latter condition is equivalent to the existence of a unital $*$-homomorphism
\[ \mathcal{D} \to (A_\infty \cap A')/A^\perp. \]
The reverse implication is obvious.

Conversely, suppose that for each $i$,
\[ \phi_i: \mathcal{D}_i \to (A_\infty \cap A')/A^\perp \]
is a $*$-homomorphism.
It suffices, by a diagonalization argument (see Section \ref{DiagArgSect}), to find u.c.p.\ maps 
\[ \mathcal{D} \to (A_\infty \cap A')/A^\perp \]
which are (point-norm-)approximately multiplicative.
Therefore, let $\mathcal{F} \subset \mathcal{D}$ be a finite subset and $\e > 0$ be our tolerance.
By density, without loss of generality, $\mathcal{F} \subset \bigcup_i \mathcal{D}_i$, which is to say that $\mathcal{F} \subset \mathcal{D}_{i_0}$ for some $i_0$.

Since $\mathcal{D}_{i_0}$ is nuclear, so is the map $\phi_{i_0}$, and therefore, there exists a u.c.p.\ map $\psi$ which factors through a finite dimensional algebra, and which approximates $\phi_{i_0}$ sufficiently well so that
\begin{equation}
 \psi(x)\psi(y) \approx_\e \psi(xy)
\alabel{LimitDStabilityCharacterization-ApproxMult}
\end{equation}
for all $x,y \in \mathcal{F}$.
Since $\psi$ factors through a finite dimensional algebra, the Arveson Extension Theorem implies that it extends to a u.c.p.\ map (which will also be denoted $\psi$)
\[ \mathcal{D} \to (A_\infty \cap A')/A^\perp. \]
This map is sufficiently approximately multiplicative, by \eqref{LimitDStabilityCharacterization-ApproxMult}.
\end{proof}

We shall now recall what the Jiang-Su algebra is.

\begin{defn}
For $p,q \in \N$, we denote
\[ \jsZ_{p,q} := \{ f \in C([0,1], M_p \tens M_q): f(0) \in M_p \tens 1_q \text{ and } f(1) \in 1_p \tens M_q\}. \]
Such a $C^*$-algebra is called a \textbf{dimension drop algebra}.
\end{defn}

\begin{thm}\cite{JiangSu}
\alabel{JiangSuCharacterization}
$\jsZ$ is (up to isomorphism) the unique simple inductive limit of dimension drop algebras which has a unique trace and whose only projections are $0$ and $1$.
In fact, $\jsZ$ is an inductive limit of dimension drop algebras of the form $\jsZ_{p-1,p}$.
\end{thm}

Note that by \ccite{Theorem 6.2.2}{EilersLoringPedersen:NCCW}, dimension drop algebras are semiprojective, so that Proposition \ref{LimitDStabilityCharacterization} applies to the above inductive limit description of $\jsZ$.

We recall the following presentation of $\jsZ_{p-1,p}$.
In the following, a ``$C^*$-algebra generated by a c.p.c.\ order zero map $\Phi:M_n$'' means a $C^*$-algebra $A$ together with a c.p.c.\ order zero map $\Phi:M_n \to A$ such that $A = C^*(\Phi(M_n))$.

\begin{prop} (R\o rdam-Winter) \alabel{DimDropPresentation}
The dimension drop algebra $\jsZ_{p-1,p}$ is the universal unital $C^*$-algebra generated by a c.p.c.\ order zero map $\Phi:M_p$ together with an element $v$ such that
\[ v^*v = (1-\Phi(1_p)) \quad \text{and} \quad v = \Phi(e_{11})v. \]
\end{prop}

\begin{proof} This is a reformulation of \ccite{Proposition 5.1 (iii) $\Leftrightarrow$ (iv)}{RordamWinter:Z}. \end{proof}

\section{Finite nuclear dimension implies divisibility}
\alabel{DimNucDivSec}

The following generalizes \ccite{Proposition 4.4}{Winter:pure}

\begin{lemma}
\alabel{DimNucCovering}
Let $A$ be a simple, separable, nonelementary $C^*$-algebra with $\dn A \leq m < \infty$
(or more generally, $A$ may be separable with $\dn A \leq m < \infty$ and such that every quotient of every ideal of $A$ is nonelementary).
Let $e \in (A_\infty)_+$ be a positive contraction and let $X \subset A_\infty$ be a separable subspace.

Then for any $k \in \N$, there are c.p.c.\ order zero maps
\[ \psi^{(1)},\dots,\psi^{((m+1)^2)}:M_k \dsum M_{k+1} \to A_\infty \cap A' \cap X' \]
such that
\[ \sum_{j=1}^{(m+1)^2} \psi^{(j)}(1_k \dsum 1_{k+1}) \geq e. \]
\end{lemma}

\begin{proof}
Let $e = \pi_\infty(e_n)$ for some sequence $(e_n)$ of positive contractions.
The proof of \ccite{Proposition 4.3}{WinterZacharias:NucDim} shows that we can find 
\[ \phi^{(1)}, \dots, \phi^{((m+1)^2)}:M_k \dsum M_{k+1} \to A_\infty \cap A' \]
such that
\[ \sum_{j=1}^{(m+1)^2} \phi^{(j)}(1_k \dsum 1_{k+1}) \geq \pi_\infty(e_n) \]
(note that the condition that every quotient of every ideal of $A$ is nonelementary is equivalent to the condition that no hereditary subalgebra of $A$ has a finite dimensional representation).
In the proof of \ccite{Proposition 4.4}{Winter:pure}, a subsequence argument is used (in the case that $e_n=1$) to get maps that additionally commute with a separable subset $X$.
The argument works here, as long as $(e_n)$ is increasing.
However, given any contractive $\pi_\infty(e_n)$ we can find an increasing approximate unit $(f_n)$ such that $\pi_\infty(f_n)\pi_\infty(e_n) = \pi_\infty(e_n)$.
Therefore, $\pi_\infty(e_n) \leq \pi_\infty(f_n)$, and it suffices to run the argument with $(f_n)$ in place of $(e_n)$.
\end{proof}

The following lemma is used both here and later in Section \ref{DivCentralSec}.
It abstracts arguments found in the proofs of \ccite{Proposition 3.6}{Winter:drZstable}, and \ccite{Propositions 4.6, 5.3}{Winter:pure}.

\begin{lemma}
\alabel{ManyOrthogTraceMaintainers}
Let $A$ be a $C^*$-algebra and let $B \subseteq A_\infty$ be a separable subalgebra.
Suppose that $k,\ell \in \mathbb{N}$ and $\gamma > 0$ are such that, for any separable commutative subalgebra $C \subseteq A_\infty \cap B'$ of nuclear dimension at most $1$, there exist c.p.c.\ order zero maps $\psi_i:M_{\ell_i} \to A_\infty \cap (B \cup C)'$ for $i=1,\dots,k$ such that
\[ \tau(\sum_{i=1}^k \psi_i(1_{\ell_i})b) \geq \gamma\tau(b) \]
for all $b \in C^*(B \cup C)_+$ and all $\tau \in T^1_\infty(A)$, and such that $\ell_i \in [2k,\ell]$ for each $i$.

Then for any $p \in \mathbb{N}$, there exist pairwise orthogonal contractions $d_1,\dots,d_{2^p} \in A_\infty \cap B'$ such that
\[ \tau(d_ib) \geq \left(\frac\gamma{k\ell}\right)^p\tau(b) \]
for all $b \in B_+$, all $i=1,\dots,2^p$, and all $\tau \in T^1_\infty(A)$.
\end{lemma}

\begin{proof}
We prove this by induction.
For the case $p=0$, we may take $d_1 \in A_\infty$ to be any positive contraction which acts as a unit on $B$.

For the inductive step, let $p \geq 0$, and we begin with $d_1,\dots,d_{2^p}$ such that
\[ \tau\left(d_i b\right) \geq \left(\frac{\gamma}{k\ell}\right)^p \tau(b) \]
for all $b \in B_+$ and all $\tau \in T^1_\infty(A)$.
We want to construct $2^{p+1}$ elements.

We set $C = C^*(d_1,\dots,d_{2^p})$, which is a commutative $C^*$-algebra of nuclear dimension at most $1$, and which commutes with $B$.
Therefore, let $\psi_i:M_{\ell_i} \to A_\infty \cap (B \cup C)'$ be as given in the hypothesis.
We will use a diagonal sequence argument (see Section \ref{DiagArgSect}, and in particular, condition \ref{DA-trace}).

Let $\gamma_0 < \frac\gamma{k\ell}$ and set $\eta := \frac\gamma{k\ell} - \gamma_0$.
Define $a := \sum_i \psi_i(e_{11}) \in A_\infty \cap (B \cup C)'$, and note that
\[ \tau(ab) \geq \frac\gamma\ell \tau(b) \]
for all $b \in C^*(C \cup B)_+$ and all $\tau \in S$ (since $\tau(\psi_i(\cdot)b)$ defines a tracial functional on $M_{\ell_i}$).

Let $e \in (A_\infty)_+$ be a positive contraction which acts as a unit on $(B \cup C \cup \{a\})$.
Define
\[ d_1' := g_{\eta,2\eta}(a) \]
and
\[ d_2' := e - g_{0,\eta}(a). \]

Using the fact that $e$ acts as an identity on $a$, we see that $d_2' \geq 0$ and $d_1'd_2' = 0$.
For $b \in C^*(B \cup C)_+$, we have $d_1'b \geq \left(\frac{a}{\|a\|} - \eta\right)b$,
so that for $\tau \in S$,
\begin{equation}
\alabel{ManyOrthogTraceMaintainers-d1ineq}
\tau(d_1'b) \geq \left(\frac{\gamma} {\ell \|a\|} - \eta\right) = \gamma_0\tau(b).
\end{equation}

Also, $b - d_2'b = (e-d_2')b = g_{0,\eta}(a)b$
so that
\begin{align*}
\tau(b) - \tau(d_2'b) &= \tau( g_{0,\eta}(a) b ) \\
&\leq \lim_{n \to \infty} \tau( a^{1/n} b ) \\
&\leq \sum_i \lim_{n \to \infty} \tau(\psi_i(e_{11})^{1/n} b) \\
&= \sum_i \lim_{n \to \infty} \tau( (\psi_i)^{1/n}(e_{11}) b) \\
&= \sum_i \lim_{n \to \infty} \frac1{\ell_i} \tau( (\psi_i)^{1/n}(1_{\ell_i}) b) \\
&\leq \left(\sum_i \frac{1}{2k}\right)\tau(b) \\
&\leq \frac12 \tau(b).
\end{align*}
In particular,
\begin{equation}
\alabel{ManyOrthogTraceMaintainers-d2ineq}
\tau(d_2'b) \geq \frac12 \tau(b) \geq \gamma_0\tau(b).
\end{equation}

Using commutativity, $(d'_id_j)_{i=1,2,j=1,\dots,2^p}$ is a family of $2^{p+1}$ orthogonal positive contractions inside $A_\infty \cap B'$.
By \eqref{ManyOrthogTraceMaintainers-d1ineq} and \eqref{ManyOrthogTraceMaintainers-d2ineq}, each member $d'_id_j$ of this family satisfies
\[ \tau(d'_id_jb) \geq \gamma_0\tau(d_jb) \geq \gamma_0\left(\frac\gamma{k\ell}\right)^p \tau(b), \]
for all $b \in B_+$ and all $\tau \in T^1_\infty(A)$.
Since $\gamma_0$ can be chosen arbitrarily close to $\frac\gamma{k\ell}$, the result follows by the diagonal sequence argument.
\end{proof}

\begin{comment}
The following is a generalization of \ccite{Proposition 4.6?}{Winter:pure}

\begin{prop}
\alabel{DimNucManyOrthog}
Given $m,k \in \N$, there is $L_{k,m} \in \N$ such that the following holds:

Suppose that $A$ is separable with $\dn \leq m$ and such that no ideal has an elementary quotient.
If $X \subset A_\infty$ is a separable subspace, then there are pairwise orthogonal positive contractions
\[ d^{(1)},\dots,d^{(k)} \in A_\infty \cap A' \cap X' \]
such that
\[ \tau(d^{(i)}b) \geq \frac1{L_{k,m}} \tau(b) \]
for all $b \in C^*(A \cup X)$ and all $\tau \in T^1_\infty(A)$.
\end{prop}

\begin{proof}
This is a direct consequence of Lemma \ref{ManyOrthogTraceMaintainers} and Lemma \ref{DimNucCovering}.
\end{proof}
\end{comment}

Here is a generalization of \ccite{Proposition 4.7?}{Winter:pure}

\begin{thm}
\alabel{DimNucTrAD}
Given $m$ there exists $\tilde{m}$ such that the following holds.
If $A$ is algebraically simple, separable, nonelementary, and with nuclear dimension at most $m$, then $A$ is tracially $\tilde m$-almost divisible.
\end{thm}

\begin{proof}
Set 
\[ p := \lceil 2 \log_2 (m+1) \rceil, \]
and
\[ \tilde{m} := \left(2(m+1)^4 + (m+1)^2\right)^p-1. \]

%In case $T^1(A) = \emptyset$, the result immediately follows from Lemma \ref{DimNucCovering}.
%Otherwise, the requirement that the map $\psi$ is nonzero in the definition of tracial almost-divisibility is automatically satisfied when $\e$ is sufficiently small.
In case $T(A) = \emptyset$, $A$ is automatically tracially $\tilde{m}$-almost divisible, see Remark \ref{DivNonzeroRmk}.

For $T(A) \neq \emptyset$, we will approximately verify the condition in Lemma \ref{SequenceAlgDiv} and appeal to the diagonal sequence argument (Section \ref{DiagArgSect}).
(As in Remark \ref{DivNonzeroRmk}, the ``strictly nonzero'' part of Lemma \ref{SequenceAlgDiv} holds automatically in this case.)
Let $\e > 0$ be our tolerance.
Let $d,e \in (A_\infty)_+$ be positive contractions such that $ed=d$ and $k \in \N$.
Pick some $\overline{k} \in \N$ satisfying
\[ \frac1{\overline{k}+1} < \e \quad \text{and} \quad k\big|\overline{k}. \]

Use Lemma \ref{DimNucCovering} to get c.p.c.\ order zero maps $\phi^{(r)}:M_{\overline{k}} \dsum M_{\overline{k}+1} \to A_\infty \cap \{d\}'$, $r=1,\dots,(m+1)^2$ such that
\[ \sum_{r=1}^{(m+1)^2} \phi^{(r)}(1_k \dsum 1_{k+1}) \geq e. \]
We now wish to use Lemma \ref{ManyOrthogTraceMaintainers} with $k=(m+1)^2$, $\ell=2(m+1)^2+1$, and $\gamma=1$ (note that this makes $k_0 = \tilde{m}+1$).
Lemma \ref{DimNucCovering} immediately verifies the hypotheses of Lemma \ref{ManyOrthogTraceMaintainers}.
Therefore by Lemma \ref{ManyOrthogTraceMaintainers}, we obtain pairwise orthogonal contractions
\[ d^{(1)},\dots,d^{((m+1)^2)} \in A_\infty \cap \{d\}' \cap \left(\bigcup_r \phi^{(r)}(M_{\overline{k}} \dsum M_{\overline{k}+1})\right)' \]
such that
\begin{equation}
\alabel{DimNucTrAD-dTauIneq}
 \tau(d_ic) \geq \frac1{\tilde m+1}\tau(c)
\end{equation}
for all $c \in C^*\left(\{d\} \cup \bigcup_r \phi^{(r)}(M_{\overline{k}} \dsum M_{\overline{k}+1})\right)$.

Define $\overline{\phi}:M_{\overline{k}} \dsum M_{\overline{k}+1} \to A_\infty$ by
\[ \overline{\phi}(x) := \sum_{r} \phi^{(r)}(x)d^{(r)}d. \]
Since this is a sum of c.p.c.\ maps with orthogonal ranges, it is c.p.c.\ order zero; it clearly satisfies $e\overline{\phi}(x) = \overline{\phi}(x)$ for all $x \in M_{\overline{k}} \dsum M_{\overline{k}+1}$.
Moreover, for $\tau \in T^1_\infty(A)$,
\begin{align*}
\tau(\overline{\phi}(1_{\overline{k}} \dsum 1_{\overline{k}+1})) &= \sum_r \tau\left(\phi^{(r)}(1_{\overline{k}} \dsum 1_{\overline{k}+1})dd^{(r)}\right) \\
&\geq^{\eqref{DimNucTrAD-dTauIneq}} \frac1{\tilde{m}+1} \sum_r \tau\left(\phi^{(r)}(1_{\overline{k}} \dsum 1_{\overline{k}+1})d\right) \\
&\geq \frac1{\tilde{m}+1}\tau(d).
\end{align*}
Since $\overline{\phi}$ is order zero, it follows that
\[ \tau(\overline{\phi}(1_{\overline{k}} \dsum 1_{\overline{k}})) \geq \frac1{\tilde{m}+1}\left(1-\frac1{\overline k + 1}\right)\tau(d) > \frac{1-\e}{\tilde{m}+1} \]
for all $\tau \in T^1_\infty(A)$.
Define $\phi:M_k \to A_\infty \cap \{d\}'$ by
\[ \phi(x) := \overline{\phi}(x \tens 1_{(\overline k/k)} \dsum x \tens 1_{(\overline k/k)}). \]
Then $e\phi(x) = \phi(x)$ for all $x \in M_k$, and $\tau(\phi(1)) \geq \frac{1-\e}{\tilde m+1}\tau(d)$ for all $\tau \in T^1_\infty(A)$.
This verifies the condition of Lemma \ref{SequenceAlgDiv} up to a tolerance $\e$; hence the result follows by the diagonal sequence algebra.
\end{proof}

\section{Tracially large, approximately central matrix cones}
\alabel{DivCentralSec}

The main result of this section shows that, in the presence of locally finite nuclear dimension, a divisibility property like that given in Lemma \ref{SequenceAlgDiv} implies a stronger one, where in particular, the range of the order zero map commutes with the algebra $A$ (and the trace-scaling constant becomes $1$).
The arguments have evolved from those found in \ccite{Section 5}{Winter:pure}.

\begin{lemma}
\alabel{CommutingDimNuc}
Let $A$ be a $C^*$-algebra and let $B,C \subseteq A$ be subalgebras such that $[B,C] = 0$.
Then
\[ \dn C^*(B \cup C) \leq (\dn B + 1)(\dn C + 1) - 1; \]
\end{lemma}

\begin{proof}
Noting that $C^*(B \cup C)$ is a quotient of an ideal of $B^\sim \tens C^\sim$ (where $B^\sim$, $C^\sim$ denote the unitizations of $B,C$),
the result holds by \ccite{Proposition 2.3 (ii) and (iv), Proposition 2.5, and Remark 2.11}{WinterZacharias:NucDim}.
\end{proof}

\begin{lemma}
\alabel{ManyMapsCommute}
Let $A$ be a $C^*$-algebra and let $B \subset A_\infty$ be a separable subalgebra of nuclear dimension at most $m$.
Let $d_0,\dots,d_m,e_0,\dots,e_m \in (A_\infty)_+ \cap B'$ and $\gamma > 0$ be such that
\begin{equation}
 \tau(d_i b) \geq \gamma \tau(b)\quad \forall b \in B_+, i=0,\dots,m, \tau \in T^1_\infty(A),
\alabel{ManyMapsCommute-dTraceIneq}
\end{equation}
and
\[ e_id_i = d_i \quad \forall i=0,\dots,m. \]
Let $\dl > 0$, and suppose that either:
\begin{enumerate}
\item \alabel{ManyMapsCommute-abel} $B$ is commutative and $A$ has tracial $\overline m$-almost divisibility, where $\dl=1/(\overline m+1)$; or
\item \alabel{ManyMapsCommute-cent} $A$ satisfies:
%has c.p.c.\ relative $\dl$-almost divisibility: 
for any c.p.c.\ order zero map $\psi:F \to A_\infty$ and any $d \in (A_\infty \cap \psi(F)')_+$, there exists a c.p.c.\ order zero map $\phi:M_k \to A_\infty \cap \psi(F)' \cap \{d\}'$ such that
\[ \tau(\phi(1_k)\psi(x)) \geq \dl\tau(\psi(x)), \quad \forall x \in F_+, \tau \in T^1_\infty(A). \]
\end{enumerate}
Then there exist $\Phi_0,\dots,\Phi_m:M_k \to A_\infty \cap B'$ such that
\begin{equation}
 \tau\left( \sum_{i=0}^m \Phi_i(1_k) b\right) \geq \gamma\dl \tau(b)
\alabel{ManyMapsCommute-FinalTrIneq}
\end{equation}
for all $b \in B_+$ and all $\tau \in T^1_\infty(A)$, and
\[ e_i\Phi_i(x) = \Phi_i(x) \]
for all $x \in M_k$.
\end{lemma}

\begin{proof}
By the diagonal sequence argument (in Section \ref{DiagArgSect}), it suffices to arrange approximate commutativity and that the trace inequality only approximately holds.
Therefore, let us fix a finite subset $\mathcal{F} \subset B$ and a tolerance $\e > 0$.
Without loss of generality, $\mathcal{F}$ contains only positive elements.

\ref{ManyMapsCommute-abel}
Set $\eta < \min\{\e/(m+1), \e/2\}$.
Let $B \iso C_0(X)$ for some space $X$.
Let $(x_j^{(i)})_{i=0,\dots,m; j=1,\dots,r} \subset X$ and $(f_j^{(i)})_{i=0,\dots,m; j=1,\dots,r}, (g_j^{(i)})_{i=0,\dots,m; j=1,\dots,r} \subset B_+$ be such that:
\begin{equation}
% \sum_{i,j} e^{(i)}_j b \approx_\eta 
b \approx_\eta \sum_{i,j} b(x_j^{(i)})f^{(i)}_j,
\alabel{ManyMapsCommute-Pf-ApproxPOU}
\end{equation}
for all $b \in F$;
\begin{equation}
 b(x) \approx_\eta b(x_j^{(i)}),
\alabel{ManyMapsCommute-Pf-SmallSupport}
\end{equation}
for all $x \in F$ and $x \in \supp g_j^{(i)}$;
\[ g^{(i)}_jf^{(i)}_j = f^{(i)}_j, \]
for all $i,j$; and for each $i$,
\[ f^{(i)}_1,\dots,f^{(i)}_r \]
are pairwise orthogonal.

By Lemma \ref{SequenceAlgDiv}, let $\Phi^{(i)}_j:M_k \to A_\infty$ be c.p.c.\ order zero maps satisfying
\[ e_ig^{(i)}_j\Phi^{(i)}_j(x) = \Phi^{(i)}_j(x) \]
for all $i,j$ and all $x \in M_k$; and
\begin{equation}
 \tau(\Phi^{(i)}_j(1_k)) \geq \dl \tau(d_if^{(i)}_j).
\alabel{ManyMapsCommute-Pf-TauIneq1}
\end{equation}
Note that \eqref{ManyMapsCommute-Pf-SmallSupport} implies that
\[ zb \approx_\eta b(x^{(i)}_j)z \]
for any contraction $z \in \her{g^{(i)}_j}$, and in particular,
\begin{equation}
 \Phi^{(i)}_j(x)b \approx_\eta b(x^{(i)}_j)\Phi^{(i)}_j(x)
\alabel{ManyMapsCommute-Pf-AlmostConstant}
\end{equation}
for any contraction $x \in M_k$.

Set
\[ \Phi^{(i)} := \sum_{j=1}^r \Phi^{(i)}_j:M_k \to A_\infty. \]
Since $\Phi^{(i)}$ is a sum of c.p.c.\ order zero maps with pairwise orthogonal ranges, it is itself c.p.c.\ order zero.

Let us compute, for $b \in \mathcal{F}$, and any contraction $x \in M_k$
\begin{align}
\notag \Phi^{(i)}(x)b &= \sum_{j=1}^r \Phi^{(i)}_j(x)b \\
&\approx_\eta \sum_{j=1}^r b(x^{(i)}_j)\Phi^{(i)}_j(x);
\alabel{ManyMapsCommute-Pf-AlmostConstantSum} 
\end{align}
for this approximation, we used \eqref{ManyMapsCommute-Pf-AlmostConstant} and the fact that the summands are orthogonal.
To see that the summands are, in fact, orthogonal, note that since $b$ commutes with $g^{(i)}_j$,
\[ b\Phi^{(i)}_j(x) \in \her{g^{(i)}_j} \perp \her{g^{(i)}_{j'}}\quad \text{for } j' \neq j. \]

From \eqref{ManyMapsCommute-Pf-AlmostConstantSum}, we obtain first that 
\[ \left\|\left[b,\Phi^{(i)}(x)\right]\right\| \leq 2\eta < \e. \]
We also find that, for $\tau \in T^1_\infty(A)$,
\begin{align*}
\tau(\Phi^{(i)}(1_k)b ) &\geq \sum_{j=1}^r b(x^{(i)}_j) \tau(\Phi^{(i)}_j(1_k)) - \eta \\
&\geq^{\eqref{ManyMapsCommute-Pf-TauIneq1}} \delta \sum_{j=1}^r b(x^{(i)}_j) \tau(d_if^{(i)}_j) - \eta \\
&\geq^{\eqref{ManyMapsCommute-dTraceIneq}} \gamma\delta \sum_{j=1}^r b(x^{(i)}_j) \tau(f^{(i)}_j) - \eta.
\end{align*}
Therefore,
\begin{align*}
\tau\left( \sum_{i=0}^m \Phi^{(i)}(1_k)b \right)
&\geq \gamma\delta \tau\left( \sum_{i,j} b(x^{(i)}_j) f^{(i)}_j\right) - m\eta \\
&\geq^{\eqref{ManyMapsCommute-Pf-ApproxPOU}} \gamma\delta \tau\left( b \right) - (m+1)\eta, 
\end{align*}
as required.

\ref{ManyMapsCommute-cent}
The proof of \ref{ManyMapsCommute-cent} is quite similar to that of \ref{ManyMapsCommute-abel}.
Set $\eta = \e/(2m+1)$.
By Lemma \ref{GoodDnApproximations}, let $e \in B_+$ be such that $eb \approx_\eta b$ for $b \in \mathcal{F}$, and let
\[ (F = F^{(0)} \dsum \cdots \dsum F^{(m)}, \psi, \phi) \]
be an $(m+1)$-colourable c.p. approximation for $B$ such that
%be a $(n+1)$-colourable approximation for $B$, which is to say that each $F_i$ is a finite dimensional $C^*$-algebra, $\beta$ is c.p.c., $\alpha$ is a c.p.\ map for which $\alpha|_{F_i}$ is c.p.c.\ order zero for each $i$, and
\begin{equation}
\alabel{ManyMapsCommute-Pf-cpApprox}
 b \approx_\eta \alpha\beta(b)
\end{equation}
and
\begin{equation}
 \alpha_i\beta_i(e)b \approx_{2\eta} \alpha_i\beta_i(b)
\alabel{ManyMapsCommute-Pf-ApproxCutdown}
\end{equation}
for $b \in \mathcal{F}$ and $i=0,\dots,m$.

For each $i$, let us apply the hypothesis of \ref{ManyMapsCommute-cent} to the c.p.c.\ order zero map $d_i\alpha_i(\cdot)$ and with $d=d_i$, to obtain 
\[ \hat\Phi^{(i)}:M_k \to A_\infty \cap (d_i\alpha_i(F_i))' \cap \{d_i\}' = A_\infty \cap \alpha_i(F_i)' \cap \{d_i\}' \]
such that
\begin{equation}
 \tau(\hat\Phi^{(i)}(1_k)d_i\alpha_i(x)) \geq \dl\tau(d_i\alpha_i(x)), \quad \forall x \in (F_i)_+, \tau \in T^1_\infty(A).
\alabel{ManyMapsCommute-Pf-TauIneq2}
\end{equation}

Set $\Phi^{(i)} = d_i\alpha_i\beta_i(e)\hat\Phi^{(i)}(\cdot)$, which is c.p.c.\ order zero since $d_i\alpha_i\beta_i(e)$ commutes with the image of $\hat\Phi^{(i)}$.
Then we have, for any contraction $x \in M_k$ and any $b \in \mathcal{F}$,
\begin{align}
\notag
\Phi^{(i)}(x)b &= \hat\Phi^{(i)}(x)d_i\alpha_i\beta_i(e)b \\
\alabel{ManyMapsCommute-Pf-Approx1}
&\approx_{2\eta}^{\eqref{ManyMapsCommute-Pf-ApproxCutdown}} \hat\Phi^{(i)}(x) d_i\alpha_i(\beta_i(b)).
\end{align}

From this we see first that $\|[\Phi^{(i)}(x),b]\| \leq 4\eta$ (since the right-hand side of \eqref{ManyMapsCommute-Pf-Approx1} is self-adjoint).
Secondly, we find that
\begin{align*}
\tau\left(\Phi^{(i)}(x)b\right) &\geq
\tau(\hat\Phi^{(i)}(x) d_i\alpha_i(\beta_i(b))) - 2\eta \\ 
&\geq^{\eqref{ManyMapsCommute-Pf-TauIneq2}} \dl \tau(d_i\alpha_i(\beta_i(b))) - 2\eta \\
&\geq^{\eqref{ManyMapsCommute-dTraceIneq}} \dl \gamma \tau(\alpha_i\beta_i(b))- 2\eta. 
\end{align*}
Summing this over all $i$ gives
\begin{align*}
\tau\left(\sum_{i=0}^m \Phi^{(i)}(x)b\right) &\geq 
\dl \gamma \tau\left(\sum_{i=0}^m \alpha_i\beta_i(b)\right) - 2m\eta \\
&= \dl \gamma \tau\left(\alpha\beta(b)\right) - 2m\eta \\
&\geq_{\eqref{ManyMapsCommute-Pf-cpApprox}} \dl \gamma \tau(b) - (2m+1)\eta,
\end{align*}
as required.
\end{proof}

\begin{lemma}
\alabel{ManyOrthogContractions}
Let $A$ be a $C^*$-algebra and let $B \subseteq A_\infty$ be a separable subalgebra of nuclear dimension at most $m$.
Let $\dl > 0$ such that either condition \ref{ManyMapsCommute-abel} or \ref{ManyMapsCommute-cent} of Lemma \ref{ManyMapsCommute} holds.
Then for $p \geq 0$, there exist orthogonal contractions $d_1,\dots,d_{2^p} \in A_\infty \cap B'$ such that
\[ \tau\left(d_i b\right) \geq \left(\frac{\dl}{(m+2)^2}\right)^p \tau(b) \]
for all $b \in B_+$ and all $\tau \in T^1_\infty(A)$.
\end{lemma}

\begin{proof}
We may find $e,d \in A_\infty$ such that $e$ acts as a unit on $d$ and $d$ acts as a unit on $B$ (since $B$ is separable).
We use Lemma \ref{ManyMapsCommute} with $d_i=d$ and $e_i=e$ for all $i$, then use its conclusion as the hypothesis to Lemma \ref{ManyOrthogTraceMaintainers} to get the conclusion of this lemma.
\end{proof}

\begin{lemma}
\alabel{OneMapTool}
Let $A$ be a $C^*$-algebra and let $B \subseteq A_\infty$ be a separable subalgebra of nuclear dimension at most $m$.
Let $\dl > 0$ such that either condition \ref{ManyMapsCommute-abel} or \ref{ManyMapsCommute-cent} of Lemma \ref{ManyMapsCommute} holds.
Then there exists a c.p.c.\ order zero map $\phi:M_k \to A_\infty \cap B'$ such that
\begin{equation}
 \tau(\phi(1_k)b) \geq \gamma_{m,\dl} \tau(b)
\alabel{OneMapTool-TraceIneq}
\end{equation}
for all $b \in B_+$ and all $\tau \in T^1_\infty(A)$, where $\gamma_{m,\dl}$ is the constant
\begin{equation}
 \gamma_{m,\dl} := \dl\left(\frac\dl{(m+2)^2}\right)^{\lceil \log_2(m+1) \rceil}.
\alabel{OneMapTool-GammaDefn}
\end{equation}
\end{lemma}

\begin{proof}
By a diagonal sequence argument (see Section \ref{DiagArgSect}), it suffices to show that \eqref{OneMapTool-TraceIneq} holds for any $\gamma' < \gamma_{m,\dl}$ in place of $\gamma_{m,\dl}$.
Given such a $\gamma'$, set 
\[ \eta := \frac{\gamma_{m,\dl}-\gamma'}\dl > 0. \]

By Lemma \ref{ManyOrthogContractions}, there exist orthogonal positive contractions $\overline d_0,\dots,\overline d_m \in A_\infty \cap B'$ such that
\[ \tau(\overline d_ib) \geq \gamma_{m,\dl}\tau(b) \]
for all $b \in B_+,i=0,\dots,m$ and all $\tau \in T^1_\infty(A)$.
Let $d_i = (\overline d_i-\eta)_+$ and $e_i = g_{0,\eta}(d_i)$, so that $e_id_i = d_i$.
We then have that
\begin{align*}
 \tau(d_ib) &\geq \tau((\overline{d_i} - \eta)b) \\
&\geq \left(\gamma_{m,\dl}-\eta\right) \tau(b) \\
&= \frac{\gamma'}\dl \tau(b).
\end{align*}
for all $b \in B_+,i=0,\dots,m$ and all $\tau \in T^1_\infty(A)$.

We may then use these with Lemma \ref{ManyMapsCommute}, with $\dl$ as already provided and with
\[ \gamma = \frac{\gamma'}\dl, \]
to get c.p.c.\ order zero maps $\phi_i:M_k \to A_\infty \cap B'$ for $i=0,\dots,m$ such that
\[ \textstyle{\tau\left(\sum_{i=0}^m \phi_i(1_k)b \right) \geq \gamma' \tau(b)} \]
for all $b \in B_+,i=0,\dots,m$ and all $\tau \in T^1_\infty(A)$, and such that $e_i\phi_i(x) = \phi_i(x)$ for all $x \in M_k$.

Since $e_0,\dots,e_m$ are orthogonal, it follows that $\phi_0,\dots,\phi_m$ have orthogonal ranges.
Thus, 
\[ \phi := \sum_{i=0}^m \phi_i \]
is itself c.p.c.\ order zero.
\end{proof}

The following builds on an argument appearing in the proof of \ccite{Lemma 5.11}{Winter:pure}.

\begin{lemma}
\alabel{GeomArg}
Let $A$ be a separable, algebraically simple $C^*$-algebra.
For a set $\mathcal{B}$ of subalgebras of $A_\infty$ and $k \in \N$, set $\beta_{\mathcal{B},k}$ to be the maximum of $\beta > 0$ such that, for all $B \in \mathcal{B}$, there exists a c.p.c.\ order zero map $\phi:M_k \to A_\infty \cap B'$ such that
\begin{equation}
\alabel{GeomArg-TrIneq}
\tau(\phi(1_k)b) \geq \beta\tau(b) \quad \forall b \in B_+.
\end{equation}
Let $\mathcal{B}$ be a set of subalgebras of $A_\infty$ and let $\mathcal{B}'$ consist of all subalgebras of $A_\infty$ of the form $C^*(B \cup C)$ where $B \in \mathcal{B}$, $C$ is the image of an order zero map from a finite dimensional $C^*$-algebra, and $[B,C]=0$.
Then for $k \in \N$, either
\[ \beta_{\mathcal{B}',k} = 0 \quad \text{or} \quad\beta_{\mathcal{B},k}=1. \]
\end{lemma}

\begin{proof}
First, we note that $\beta_{\mathcal{B},k}$ is truly a maximum, by a diagonal sequence argument (Section \ref{DiagArgSect}).
Suppose that $\beta_{\mathcal{B}',k} \neq 0$.

Set $\eta := 1-\beta_{\mathcal{B},k} \geq 1-\beta_{\mathcal{B}',k} \geq 0$.
We shall show that
\begin{equation}
 \eta \leq \left(1-\frac{\beta_{\mathcal{B}',2}^2}k\right)\eta + \e,
\alabel{GeomArg-RedStep}
\end{equation}
for any $\e > 0$, from which it easily follows that $\eta = 0$.

To this end, let $\e > 0$ and let $B \in \mathcal{B}$.
Let $\phi:M_k \to A_\infty \cap B'$ be c.p.c.\ order zero such that \eqref{GeomArg-TrIneq} holds with $\beta=\beta_{\mathcal{B},k}$.
$C^*(B \cup \phi(M_k)) \in \mathcal{B}'$, so that there exists $\rho:M_k \to A_\infty \cap (B \cup \phi(M_k))'$ such that
\[ \tau(\rho(1_k)b) \geq \beta_{\mathcal{B}',k}\tau(b) \]
for all $b \in C^*(B \cup \phi(M_k))_+$.
Set $d_i = \rho(e_{ii})$ for $i=1,2$, so that $d_1,d_2 \in A_\infty \cap (B \cup \phi(M_k))'$ are orthogonal positive contractions satisfying
\begin{equation}
 \tau(d_ib) \geq \frac{\beta_{\mathcal{B}',k}}k\tau(b)
\alabel{GeomArg-dTraceIneq}
\end{equation}
for all $b \in C^*(B \cup \phi(M_k))_+$.

Set
\begin{equation}
h_1 := 1-g_{0,\e},\  
h_2 := g_{0,\e} - g_{\e,2\e},\ 
h_3 := g_{\e,2\e} \in C([0,1]),
\alabel{GeomArg-hDefn}
\end{equation}
which of course satisfy
\begin{equation}
 h_1+h_2+h_3 = 1.
\alabel{GeomArg-hSum1}
\end{equation}

Since $C^*(B \cup \{d_1h_1(\phi(1_k))\}) \in \mathcal{B'}$, there exists $\psi:M_k \to A_\infty \cap (B \cup  \{d_1h_1(\phi(1_k))\})'$ such that
\begin{equation}
\tau(\psi(1_k)b) \geq \beta_{\mathcal{B}',k}\tau(b),
\alabel{GeomArg-psiTraceIneq}
\end{equation}
for all $b \in C^*(B \cup  \{d_1h_1(\phi(1_k))\})_+$.

Now, notice that $d_2h_2(\phi)(\cdot) + h_3(\phi)(\cdot)$ is c.p.c.\ order zero, since it is c.p.\ and dominated by $g_{0,\e}(\phi)$.
Also, it is orthogonal to the c.p.c.\ order zero map $d_1h_1(\phi(1_k))\psi(\cdot)$.
Thus,
\[ \Phi := d_1h_1(\phi(1_k))\psi(\cdot) + d_2h_2(\phi)(\cdot) + h_3(\phi)(\cdot):M_k \to A_\infty \cap B' \]
is a c.p.c.\ order zero map.
Moreover, for $b \in B_+$ and $\tau \in T^1_\infty(A)$, we have
\begin{align*}
\tau(\Phi(1_k)b) &= \tau(d_1h_1(\phi(1_k))\psi(1_k)b + d_2h_2(\phi(1_k))b + h_3(\phi(1_k))b) \\
&\geq^{\eqref{GeomArg-dTraceIneq}, \eqref{GeomArg-psiTraceIneq}} \tau\left(\frac{\beta_{\mathcal{B'},k}^2}k h_1(\phi(1_k))b + \frac{\beta_{\mathcal{B}',k}}{k}h_2(\phi(1_k))b + h_3(\phi(1_k))b\right) \\
&\geq^{\eqref{GeomArg-hSum1}} \frac{\beta_{\mathcal{B'},k}^2}k \tau(b) + \left(1-\frac{\beta_{\mathcal{B'},k}^2}k\right)\tau(h_3(\phi(1_k))b) \\
&\geq^{\eqref{GeomArg-hDefn}} \frac{\beta_{\mathcal{B'},k}^2}k \tau(b) + \left(1-\frac{\beta_{\mathcal{B'},k}^2}k\right)\tau((\phi(1_k) - \e)b) \\
&\geq^{\eqref{GeomArg-TrIneq}} \left(\frac{\beta_{\mathcal{B'},k}^2}k + \left(1-\frac{\beta_{\mathcal{B'},k}^2}k\right)(\beta_{\mathcal{B'},k} - \e)\right)\tau(b) \\
&\geq \left(1 - \left(\left(1 - \frac{\beta_{\mathcal{B'},k}^2}k\right)\eta + \e\right)\right)\tau(b).
\end{align*}
Thus, $\Phi$ witnesses \eqref{GeomArg-RedStep} (for the given $B$), as required.
\end{proof}

\begin{thm}
\alabel{OneMap}
%Given $m,\tilde{m} \in \mathbb{N}$, there exists $\beta_{m,\tilde{m}}$ such that the following holds.
Let $A$ be a separable, algebraically simple, exact $C^*$-algebra with tracial $m$-almost-divisibility.
Let $B \subset A_\infty$ a separable subalgebra with nuclear dimension at most $\tilde m$, and let $k \in \mathbb{N}$.
Then there exists a c.p.c.\ order zero map $\Phi:M_k \to A_\infty \cap B'$ such that
\begin{equation}
\alabel{OneMap-TraceIneq}
\tau(\Phi(1_k)b) = \tau(b)
\end{equation}
for all $b \in B_+$ and all $\tau \in T^1_\infty(A)$.
\end{thm}

\begin{proof}
Let $\mathcal{B}_{\tilde m}$ be the set of all separable subalgebras of $A_\infty$ of nuclear dimension at most $\tilde m$.
Defining $\mathcal{B}_{\tilde m}'$ as in Lemma \ref{GeomArg}, we note that by Lemma \ref{CommutingDimNuc}, we have $\mathcal{B}_{\tilde m}' \subseteq \mathcal{B}_{2\tilde m +1}$.
Thus, by Lemma \ref{GeomArg}, it suffices to show that
\[ \beta_{\mathcal{B}_{\tilde m},k} > 0 \]
for all $k$ and $\tilde m$.

$A$ satisfies the hypothesis of Lemma \ref{ManyMapsCommute} \ref{ManyMapsCommute-abel}, with $\dl = 1/(\tilde{m}+1)$, and therefore Lemma \ref{OneMapTool} shows that:

\begin{enumerate}
\item[($*$)] For any separable commutative algebra $C \subset A_\infty$ of nuclear dimension at most $\ell$, there exists a c.p.c.\ order zero map $\phi:M_k \to A_\infty \cap C'$ such that
\[ \tau(\phi(1_k)c) \geq \gamma_{\ell,1/(\tilde{m}+1)} \tau(c) \]
for all $c \in C_+$ and all $\tau \in T^1_\infty(A)$.
\end{enumerate}

We wish to show that hypothesis \ref{ManyMapsCommute-cent} of Lemma \ref{ManyMapsCommute} holds.
This hypothesis shall follow (as will be explained) from strengthening ($*$), so that instead of being commutative, $C$ is allowed to be of the form
\begin{equation}
\alabel{SemiHomog}
 C = \bigdsum_{i=1}^r C_0(X_i,M_{n_i})
\end{equation}
(for some spaces $X_i$).

(As was pointed out to the author by Stuart White, Lemma \ref{GeomArg} shows that this strengthening already implies that $\gamma_{\ell,1/(\tilde m + 1)} = 1$, though of course, this also follows once this entire theorem is proven).

Let us first consider the case that $C$ has the form \eqref{SemiHomog}, except with $r=1$.
Moreover, given $e \in (A_\infty)_+$ satisfying $ec = c$ for all $c \in C$, we will arrange that
\[ e\phi(x) = \phi(x) \]
for all $x \in M_k$.

For a finite subset $\mathcal{F} \subset C_0(X_1)_+$ and a tolerance $\e > 0$, we will construct a c.p.c.\ order zero map $\phi:M_k \to A_\infty \cap C'$ such that $e\phi(x) = \phi(x)$ for all $x \in M_k$ and 
\begin{equation}
\alabel{OneMap-Pf-TraceIneq1}
\tau\left( \phi(1_k)c \tens a\right) \geq \gamma_{\ell,1/(\tilde{m}+1)} \tau(c \tens a) - \e,
\end{equation}
for any $c \in F$ and $a \in M_{n_1}$.
From this, it follows by a diagonal sequence argument (see Section \ref{DiagArgSect}) that there exists $\phi:M_k \to A_\infty \cap C'$ c.p.c.\ order zero such that $e\phi(x) = \phi(x)$ for all $x \in M_k$ and
\begin{equation}
\alabel{OneMap-Pf-TraceIneq2}
 \tau\left( \phi(1_k)d\right) \geq \gamma_{\ell,1/(\tilde{m}+1)} \tau(d),
\end{equation}
for any $d = \sum_{i=1}^p c_i \tens a_i$, where $c_1,\dots,c_p \in C_0(X_1)_+$ and $a_1,\dots,a_k \in (M_{n_1})_+$.
Since $C_0(X_1,M_{n_1}) \iso C_0(X_1) \otimes M_{n_1}$, we know that any $d \in C$ can be approximated by elements of the form just given, except that the tensor factors needn't be positive.
However, the proof that $C_0(X_1,M_{n_1}) \iso C_0(X_1) \otimes M_{n_1}$ shows that if $d \in C_+$ then it can be approximated by sums $\sum_{i=1}^p c_i \tens a_i$, where $c_1,\dots,c_p \in C_0(X_1)_+$ and $a_1,\dots,a_k \in (M_{n_1})_+$ (of course, this does not hold for tensor products of two noncommutative $C^*$-algebras).
By continuity, it will follow that \eqref{OneMap-Pf-TraceIneq2} holds for any $d \in C_+$.
%; and by continuity, also for any $d$ which can be approximated by elements of such a form.
%
%But, every $d \in C_+$ can be approximated by such elements.
%Certainly, given $\e > 0$, we may find an open cover $U_1,\dots,U_n$ of $X_1$ and points $x_1,\dots,x_n \in X_1$ such that for $y \in U_i$, $\|d(x_i)-d(y)\| < \e/2$.
%We may also find positive functions $e_1,\dots,e_n \in C_0(X_1)$ such that $e_i$ is supported on $U_i$ and
%\[ \sum e_id \approx_{\e/2} d. \]
%It then follows that
%\[ d \approx_\e \sum_{i=1}^n e_i \tens \left(d(x_i)\right). \]

Therefore, let us fix $\mathcal{F} \subset C_0(X_1)_+$ finite and $\e >0$, and proceed to construct $\phi$.
($*$) provides us with $\psi:M_k \to A_\infty \cap (C_0(X_1) \tens e_{11})'$ such that
\begin{equation}
 \tau(\psi(1_k)c \tens e_{11}) \geq \gamma_{\ell,1/(\tilde{m}+1)} \tau(c \tens e_{11})
\alabel{Fact2-TauIneq}
\end{equation}
Let $a \in C_0(X_1)$ be a positive contraction such that
\begin{equation}
 aca \approx_{\e/{n_1}} c
\alabel{Fact2-ApproxId}
\end{equation}
 for all $c \in \mathcal{F}$.
Then define $\phi:M_k \to A_\infty$ by
\[ \phi(x) := \sum_{i=1}^{n_1} (a \tens e_{i1})\psi(x)(a \tens e_{1i}). \]
Since $e$ acts as a unit on $C$, $e\phi(x) = \phi(x)$ for all $x \in M_k$, and since $\phi$ is defined as a sum of c.p.c.\ order zero maps with orthogonal images, it is c.p.c.\ order zero.
To see that the image of $\phi$ commutes with $C = \Span \{c \tens e_{ij}: c \in C_0(X_1), i,j =1,\dots,n_1\}$, we compute
\begin{align*}
\phi(x)(c \tens e_{ij}) &= \sum_{k=1}^{n_1} (a \tens e_{k1})\psi(x)(a \tens e_{1k})(c \tens e_{ij}) \\
&= (a \tens e_{i1})\psi(x)(ca \tens e_{1j}) \\
&= (a \tens e_{i1})\psi(x)(c \tens e_{11})(a \tens e_{1j}) \\
&= (ac \tens e_{i1})\psi(x)(a \tens e_{1j}) \\
&= (c \tens e_{ij})\phi(x),
\end{align*}
where on the fourth line, we have used the fact that $\psi(x)$ commutes with $c \tens e_{11}$.

To verify the trace inequality \eqref{OneMap-Pf-TraceIneq1}, let $a \in (M_{n_1})_+, c \in \mathcal{F},$ and $\tau \in T^1_\infty(A)$.
If $\lambda_1,\dots,\lambda_{n_1}$ are the eigenvalues of $a$ (with multiplicity) then there exists $v \in C$ such that
\[ v^*v = c \tens a \quad\text{and}\quad vv^* = c \tens \left(\sum_{i=1}^{n_1} \lambda_i e_{ii}\right), \]
and thus,
\begin{align*}
\tau\left(\phi(1_k)(c \tens a)\right) &= \sum_{i=1}^{n_1} \lambda_i \tau\left(\phi(1_k)(c \tens e_{ii})\right) \\
&= \sum_{i=1}^{n_1} \lambda_i \tau\left(\phi(1_k)(c \tens e_{11})\right) \\
&= \sum_{i=1}^{n_1} \lambda_i \tau((a \tens e_{11})\psi(1_k)(a \tens e_{11})(c \tens e_{11})) \\
%&= \sum_{i=1}^{n_1} \lambda_i \tau(\psi(1_k)(aca \tens e_{11})) \\
%&\geq^{\eqref{Fact2-ApproxId}} \sum_{i=1}^{n_1} \lambda_i \right(\tau(\psi(1_k)c \tens e_{11}) - \frac\e{{n_1}}\left) \\
&\geq^{\eqref{Fact2-TauIneq},\eqref{Fact2-ApproxId}} \sum_{i=1}^{n_1} \lambda_i \left( \gamma_{\ell,1/(\tilde{m}+1)}\tau(c \tens e_{11}) - \frac{\e}{n_1} \right) \\
&= \gamma_{\ell,1/(\tilde{m}+1)} \tau(c \tens a) - \sum_{i=1}^{n_1} \lambda_i\frac\e{n_1} \\
&\geq \gamma_{\ell,1/(\tilde{m}+1)} \tau(c \tens a) - \e.
\end{align*}
This concludes the verification of the fact that, for any separable algebra $C \subset A_\infty$ of the form \eqref{SemiHomog} with $r=1$, and of nuclear dimension at most $\ell$, and for any $e \in A_\infty$ such that $ec=c$ for all $c \in C$, there exists $\phi:M_k \to A_\infty \cap C'$ such that
\[ \tau(\phi(1_k)c) \geq \gamma_{\ell,1/(\tilde{m}+1)} \tau(c) \]
for all $c \in C_+$ and all $\tau \in T^1_\infty(A)$ and
\[ e\phi(x) = \phi(x) \]
for all $x \in M_k$.

Now, let $C \subset A_\infty$ be of the form \eqref{SemiHomog} with arbitrary $r$.
Once again, fixing a finite subset $\mathcal{F} \subset C_+$ and a tolerance $\e > 0$, we will construct $\phi$ whose range commutes with $\mathcal{F}$ and satisfies
\[ \tau\left( \phi(1_k)c \right) \geq \gamma_{\ell,1/(\tilde{m}+1)} \tau(c ) - \e, \]
for any $c \in \mathcal{F}$; and from this, the full result for $C \subset A_\infty$ of the form \eqref{SemiHomog}.

Without loss of generality, by approximating, we may assume that $\mathcal{F}$ consists of compactly supported functions; that is, that $\mathcal{F}$ is contained in
\[ \bigdsum_{i=1}^r C_0(Y_i,M_{n_i}) \]
where $Y_i$ is compactly contained in $X_i$.
Let $e_i \in C_0(X_i)$ be such that $e_i|_{Y_i} \equiv 1$.
We have just shown that there exist $\phi_i:M_k \to A_\infty \cap C_0(Y_i,M_{n_i})'$ for each $i$ such that $e_i\phi_i(x) = \phi_i(x)$ for all $x \in M_k$ and
\[ \tau(\phi(1_k)c) \geq \gamma_{\ell,1/(\tilde{m}+1)} \tau(c) \]
for all $c \in C_0(Y_i,M_{n_i})_+$ and all $\tau \in T^1_\infty(A)$.

Since $e_i$ acts as a unit on the range of $\phi_i$, the ranges of $\phi_1,\dots,\phi_r$ are orthogonal and therefore $\phi := \sum_{i=1}^n \phi_i$ is a c.p.c.\ order zero map.
It is also clear that the range of $\phi$ commutes with $\mathcal{F}$ and satisfies
\[ \tau(\phi(1_k)c) \geq \gamma_{\ell,1/(\tilde{m}+1)} \tau(c) \]
for all $c \in \mathcal{F}$ and all $\tau \in T^1_\infty(A)$.
This concludes the verification that we can strengthen ($*$) to the case that $C$ is of the form \eqref{SemiHomog}.

Let us use this to finally verify hypothesis \ref{ManyMapsCommute-cent} of Lemma \ref{ManyMapsCommute}, with $\dl = \gamma_{2,1/(\tilde{m}+1)}$.
Certainly, given a c.p.c.\ order zero map $\psi:F \to A_\infty$ and $d \in (A_\infty \cap \psi(F)')_+$, we know that $\psi(F)$ is a quotient of $C_0((0,1], F)$ by Proposition \ref{OrderZeroStructure}.
Since $d$ commutes with $\psi(F)$, it follows that $C := C^*(\psi(F) \cup \{d\})$ has nuclear dimension at most $2$ and is of the form \eqref{SemiHomog}.
We have just shown that there exists $\phi:M_k \to A_\infty \cap \psi(F)' \cap \{d\}'$ such that, in particular,
\[ \tau(\phi(1_k)\psi(x)) \geq \gamma_{2,1/(\tilde{m}+1)}\tau(\psi(x)) \]
for all $x \in F_+$ and $\tau \in T^1_\infty(A)$.

Applying Lemma \ref{OneMapTool} gives us $\Phi:M_k \to A \cap B'$ witnessing $\beta_{m,\tilde{m}} \geq \gamma_{m,\gamma_{2,1/(\tilde{m}+1)}} > 0$, as required.
\end{proof}

\section{Central dimension drop embeddings and proof of the main theorem}
\alabel{MainProofSect}

Roughly following the arguments of \ccite{Section 4}{Winter:drZstable}, we prove here that the conclusion of Theorem \ref{OneMap} combined with strong tracial $m$-comparison (and locally finite nuclear dimension) provide central dimension drop embeddings into the algebra described in Proposition \ref{AnnQuotientUnital}.
In \ccite{Section 6}{Winter:pure}, the arguments of \ccite{Section 4}{Winter:drZstable} were already adapted to use locally finite nuclear dimension in place of finite decomposition rank.
While certain innovations are required here to handle the nonunital case, we have also attempted to make the arguments more conceptual by making more use of the asymptotic sequence algebra.

\begin{lemma}
\alabel{DefectEqOneElement}
(cf.\ \ccite{Proposition 4.2}{Winter:drZstable})
Let $A$ be separable, algebraically simple, exact, with strict tracial $m$-comparison, and such that $T^1(A) \neq \emptyset$.
Let $a \in A_+$ be a positive contraction and let
\[ \Phi:M_n \tens M_2 \to A_\infty \cap \{a\}' \]
be a c.p.c.\ order zero map for which $\tau(\Phi\left(1_{M_n \tens M_2}\right)) = 1$ for all $\tau \in T^1_\infty(A)$.
Let $e \in (A_\infty)_+$ satisfy $ea=a$.
Then there exists $v \in A_\infty \cap \{a\}'$ such that
\begin{enumerate}
\item \alabel{DefectEqOneElement-id} $ev = ve = v$;
\item \alabel{DefectEqOneElement-mv1} $v^*v = a(1-\Phi\left(1_{M_n \tens M_2}\right))$; and
\item \alabel{DefectEqOneElement-mv2} $v = \Phi\left(e_{11} \tens 1_2\right)v$.
\end{enumerate}
\end{lemma}

\begin{proof}
Using a diagonal sequence argument (see Section \ref{DiagArgSect}), it suffices to find $v$ which approximately commutes with $a$ and which approximately satisfies conditions \ref{DefectEqOneElement-mv1}, \ref{DefectEqOneElement-mv2} (while exactly satisfying \ref{DefectEqOneElement-id}). 
Therefore, let us fix a tolerance $\e > 0$.

Let $\eta < \min\left\{\frac\e4,\left(\frac\e2\right)^2\right\}$.
There exist $\overline{t} \in \N$, functions $h^{(i)}_t \in C_0((0,1])_+$ and points $\alpha^{(i)}_t \in (0,1]$ for $i=1,2$ and $t=1,\dots,\overline{t}$ such that for $i=1,2$, $h^{(i)}_1,\dots,h^{(i)}_{\overline{t}}$ are pairwise orthogonal and
\begin{align}
\alabel{DefectEqOneElement-ApproxPOU2}
\id_{(0,1]} &\approx_\eta \sum_{i=1,2}\sum_{t=1}^{\overline{t}} \alpha^{(i)}_t h^{(i)}_t \text{ and} \\
\alabel{DefectEqOneElement-ApproxPOU1}
 f(h^{(i)}_t)\id_{(0,1]} &\approx_\eta \alpha^{(i)}_tf\left(h^{(i)}_t\right)
\end{align}
for any contractive $f \in C_0((0,1])$.
Set $a^{(i)}_t := h^{(i)}_t(a)$.

Let $f \in C_0((1-\eta^2,1])_+$ satisfy $f(1) = 1$.

Let $J = \{(i,t): a^{(i)}_t \neq 0\}$.
By Proposition \ref{Trace1}, we find that
\[ \tau\left(a^{(i)}_t(1-\Phi\left(1_{M_n \tens M_2}\right))\right) = 0, \]
for all $\tau \in T^1_\infty(A)$.
On the other hand, for $(i,t) \in J$, by Proposition \ref{PedAProperties} \ref{PedAProperties-BddBelow},
\[ \gamma^{(i)}_t := \inf_{\tau \in T^1(A)} \tau\left(a^{(i)}_t\right) > 0; \]
and we have
\begin{align*}
\tau\left(\Phi\left(e_{11} \tens e_{ii}\right)a^{(i)}_t\right) &= \frac{1}{2n} \tau\left(\Phi\left(1_{M_n \tens M_2}\right)a^{(i)}_t\right) \\
&\geq \frac1{2n} \tau\left(a^{(i)}_t\right) \geq \frac{\gamma^{(i)}_t}{2n},
\end{align*}
where the first inequality follows from the fact that $\Phi$ is order zero and its image commutes with $a^{(i)}_t$, and the second inequality follows from Proposition \ref{Trace1}.
Thus, by Lemma \ref{SequenceAlgComparison} and \ccite{Proposition 2.4}{Rordam:UHFII}, for $t \in J$, there exists $v^{(i)}_t \in A_\infty$ such that
\begin{equation}
 \left(v^{(i)}_t\right)^*v^{(i)}_t = \left(a^{(i)}_t\left(1-\Phi\left(1_{M_n \tens M_2}\right)\right) - \eta\right)_+ 
\alabel{DefectEqOneElement-vtDef1}
\end{equation}
and
\begin{equation}
 v^{(i)}_t\left(v^{(i)}_t\right)^* \in \her{a^{(i)}_tf\left(\Phi\left(e_{11} \tens e_{ii}\right)\right)}
\alabel{DefectEqOneElement-vtDef2}
\end{equation}

Note that, since
\[ a^{(i)}_tf(\Phi\left(e_{11} \tens e_ {ii}\right)) \perp a^{(i')}_{t'}f\left(\Phi\left(e_{11} \tens e_ {i'i'}\right)\right), \]
whenever $(i,t) \neq (i',t')$, it follows that $v^{(i)}_t\left(v^{(i)}_t\right)^* \perp v^{(i')}_{t'}\left(v^{(i')}_{t'}\right)^*$ and therefore,
\begin{equation}
\left(v^{(i)}_t\right)^*v^{(i')}_{t'} = 0.
\alabel{DefectEqOneElement-vtOrthog1}
\end{equation}
By the same arguments, we find that, if $i=1$ or $2$ and if $t \neq t'$ then since $a$ and $\Phi\left(e_{11} \tens 1_2\right)$ commute with $a^{(i)}_t$ and $a^{(i)}_{t'}$,
\begin{gather}
\alabel{DefectEqOneElement-vtOrthog2}
v^{(i)}_t\left(v^{(i)}_{t'}\right)^*, \left(v^{(i)}_ta\right)\left(v^{(i)}_{t'}a\right)^*, \left(av^{(i)}_t\right)^*\left(av^{(i)}_{t'}\right), \\
\notag
\left(\Phi\left(e_{11} \tens 1_2\right)v^{(i)}_t\right)^*\left(\Phi\left(e_{11} \tens 1_2\right)v^{(i)}_{t'}\right) = 0.
\end{gather}

Since $v^{(i)}_t \in \her{a^{(i)}_t}$, we see by \eqref{DefectEqOneElement-ApproxPOU1} that
\begin{equation}
\alabel{DefectEqOneElement-vtCommute}
 v^{(i)}_ta \approx_\eta \alpha^{(i)}_tv^{(i)}_t \approx_\eta av^{(i)}_t.
\end{equation}
Likewise, since $\id_{(0,1]}f \approx_\eta f$,
\[ \Phi\left(e_{11} \tens 1_2\right)v^{(i)}_t\left(v^{(i)}_t\right)^* \approx_\eta v^{(i)}_t\left(v^{(i)}_t\right)^*, \]
from which it follows that
\begin{equation}
\alabel{DefectEqOneElement-vtApproxId}
 \Phi\left(e_{11} \tens 1_2\right)v^{(i)}_t \approx_{\eta^{1/2}} v^{(i)}_t.
\end{equation}

Now we may define
\[ v := \sum_{i=1,2} \sum_{t=1}^{\overline{t}} \left(\alpha^{(i)}_t\right)^{1/2} v^{(i)}_t. \]

Let us now check that $v$ is as required.

\textbf{$v$ approximately commutes with $a$}:
We have
\begin{align*}
\left\|va - \sum_{i=1,2} \sum_{t=1}^{\overline{t}} \left(\alpha^{(i)}_t\right)^{3/2} v^{(i)}_t\right\| &\leq
\sum_{i=1,2} \left\|\sum_{t}  \left(\alpha^{(i)}_t\right)^{1/2}(v^{(i)}_ta - \alpha^{(i)}_t v^{(i)}_t)\right\| \\
&= \sum_{i=1,2} \max_{t} \left\|\left(\alpha^{(i)}_t\right)^{1/2}(v^{(i)}_ta - \alpha^{(i)}_t v^{(i)}_t)\right\| \\
&\leq^{\eqref{DefectEqOneElement-vtCommute}} 2\eta,
\end{align*}
where the second line uses the fact that the summands are orthogonal (by \eqref{DefectEqOneElement-vtOrthog1} and \eqref{DefectEqOneElement-vtOrthog2}).
Likewise, we obtain that
\[ av \approx_{2\eta}  \sum_{i=1,2} \sum_{t=1}^{\overline{t}} \left(\alpha^{(i)}_t\right)^{3/2} v^{(i)}_t, \]
so that altogether, $\|[a,v]\| \leq 4\eta < \e$.

\textbf{\ref{DefectEqOneElement-id} holds}:
Evidently, $v^{(i)}_t \in \her{a^{(i)}_t}$, and therefore, $v \in \her{a}$.
It follows that $ev=ve=v$.

\textbf{\ref{DefectEqOneElement-mv1} holds approximately}:
We compute
\begin{align*}
v^*v &= \sum_{i,i',t,t'} \left(v^{(i)}_t\right)^*v^{(i')}_{t'} \\
&=^{\eqref{DefectEqOneElement-vtOrthog1}} \sum_{i,t} \left(v^{(i)}_t\right)^*v^{(i)}_t \\
&=^{\eqref{DefectEqOneElement-vtDef1}}  \sum_{i,t} \alpha^{(i)}_t\left(a^{(i)}_t\left(1-\Phi\left(1_{M_n \tens M_2}\right)\right)-\eta\right)_+ \\
&\approx_{2\eta} \sum_{i,t} \alpha^{(i)}_ta^{(i)}_t\left(1-\Phi\left(1_{M_n \tens M_2}\right)\right) \\
&\approx^{\eqref{DefectEqOneElement-ApproxPOU2}}_\eta a\left(1-\Phi\left(1_{M_n \tens M_2}\right)\right), 
\end{align*}
where the fourth line is achieved by, once again, splitting into two orthogonal sums.

\textbf{\ref{DefectEqOneElement-mv2} holds approximately}:
We compute
\begin{align*}
\Phi\left(e_{11} \tens 1_2\right)v &= \sum_{i,t} \Phi\left(e_{11} \tens 1_2\right)v^{(i)}_t \\
&\approx_{2\eta^{1/2}} v^{(i)}_t,
\end{align*}
by once again splitting into two orthogonal sums (by \eqref{DefectEqOneElement-vtOrthog2}) and using \eqref{DefectEqOneElement-vtApproxId}.
\end{proof}

\begin{lemma}
\alabel{DefectEqCpc}
(cf. \ccite{Proposition 4.3}{Winter:drZstable})
Let $A$ be separable, algebraically simple, exact, with strict tracial $m$-comparison, and such that $T^1(A) \neq \emptyset$.
Let $F$ be a finite dimensional $C^*$-algebra,
\[ \psi:F \to A \]
be a c.p.c.\ order zero map and
\[ \Phi:M_n \tens M_2 \to A_\infty \cap \psi(F)' \]
be another c.p.c.\ order zero map for which $\tau(\Phi\left(1_{M_n \tens M_2}\right)) = 1$ for all $\tau \in T^1_\infty(A)$.
Then there exists $v \in A_\infty \cap \psi(F)'$ such that
\begin{enumerate}
\item \alabel{DefectEqCpc-mvn} $v^*v = \psi(1)(1-\Phi\left(1_{M_n \tens M_2}\right))$; and
\item \alabel{DefectEqCpc-id} $v = \Phi\left(e_{11} \tens 1_2\right)v$.
\end{enumerate}
\end{lemma}

\begin{proof}
Again, by the diagonal sequence argument (see Section \ref{DiagArgSect}), it suffices to find $v$ that only approximately satisfies (i).
Therefore let us fix a tolerance $\e$.

Let
\[ F = M_{r_1} \dsum \cdots \dsum M_{r_s} \]
and denote $\psi_i = \psi|_{M_i}$.
By Lemma \ref{DefectEqOneElement}, there exists
\[ w_i \in A_\infty \cap \{(\psi_i(e_{11})-\e)_+\}' \]
such that
\begin{enumerate}
\item $g_{0,\e}(\psi_i(e_{11}))w_i = w_ig_{0,\e}(\psi_i(e_{11})) = w_i$;
\item $w_i^*w_i = \left(\psi_i(e_{11})-\e\right)_+\left(1-\Phi\left(1_{M_n \tens M_2}\right)\right)$; and
\item $w_i = \Phi\left(e_{11} \tens 1_2\right)w_i$.
\end{enumerate}

Set
\[ v := \sum_{i=1}^s \sum_{j=1}^{r_i} g_{0,\e}(\psi_i)(e_{j1})w_ig_{0,\e}(\psi_i)(e_{1j}). \]
By testing on matrix units, we can see that $v$ commutes with $\psi(F)$.
Using the fact that $\psi(F)$ commutes with $\Phi\left(e_{11} \tens 1_2\right)$, one can see that \ref{DefectEqCpc-id} holds.
For \ref{DefectEqCpc-mvn}, we compute
\begin{align*}
v^*v &= \sum_{i,j,i',j'} g_{0,\e}(\psi_i)(e_{j1})\,w_i^*\,g_{0,\e}(\psi_i)(e_{1j})\,g_{0,\e}(\psi_{i'})(e_{j'1})\,w_{i'}\,g_{0,\e}(\psi_{i'})(e_{1j'}) \\
&= \sum_{i,j} g_{0,\e}(\psi_i)(e_{j1})\,w_i^*\,g_{0,\e}(\psi_i)(e_{1j})\,g_{0,\e}(\psi_i)(e_{j1})\,w_i\,g_{0,\e}(\psi_i)(e_{1j}) \\
&= \sum_{i,j} g_{0,\e}(\psi_i)(e_{j1})\,w_i^*\,g^2_{0,\e}(\psi_i)(e_{11})\,w_ig\,_{0,\e}(\psi_i)(e_{1j}) \\
&= \sum_{i,j} g_{0,\e}(\psi_i)(e_{j1})w_i^*w_ig_{0,\e}(\psi_i)(e_{1j}) \\
&=^{\text{Lemma \ref{DefectEqOneElement} \ref{DefectEqOneElement-id}}}  \sum_{i,j} g_{0,\e}(\psi_i)(e_{j1})(\psi_i(e_{11})-\e)_+\left(1-\Phi\left(1_{M_n \tens M_2}\right)\right)g_{0,\e}(\psi_i)(e_{1j}) \\
&= \sum_{i,j} (\psi_i(e_{jj})-\e)_+\left(1-\Phi\left(1_{M_n \tens M_2}\right)\right) \\
&= (\psi(1_F)-\e)_+\left(1-\Phi\left(1_{M_n \tens M_2}\right)\right) \\
&\approx_\e \psi(1_F)\left(1-\Phi\left(1_{M_n \tens M_2}\right)\right).
\end{align*}
\end{proof}

\begin{prop}
\alabel{vCommuteDN}
(cf.\ \ccite{Proposition 4.4}{Winter:drZstable})
Let $A$ be a simple, separable $C^*$-algebra and let $B \subseteq A$ be a subalgebra with nuclear dimension at most $m < \infty$.
Given a finite subset $\mathcal{F} \subset B$, positive contractions $g,h \in C_0((0,1])_+$, and $\e > 0$, there is a finite subset $\mathcal{G} \subset B$ and $\dl > 0$ such that the following holds:

Suppose that $e \in B_+$ is a positive contraction such that
\[ ex \approx_\dl x \approx_\dl xe \]
for all $x \in \mathcal{F}$, that $(F = F^{(0)} \dsum \cdots \dsum F^{(m)}, \sigma, \rho)$ is an $(m+1)$-colourable c.p.\ approximation (for $B$) of $\mathcal{G} \cup \{e\}$ to within $\dl$, and that $v_0,\dots,v_m \in A$ are contractions which satisfy
\begin{equation}
\alabel{vCommuteDN-LocalCommute}
\| [\rho^{(i)}(x),v_i] \| \leq \dl\|x\|
\end{equation}
for all $x \in F^{(i)}$.
Then
\begin{equation}
\alabel{vCommuteDN-vDef}
v := \sum_{i=0}^m v_ig(h(\rho^{(i)}\sigma^{(i)}(e)))
\end{equation}
satisfies
\[ \|[v,a]\| < \e \]
for all $a \in \mathcal{F}$.
\end{prop}

\begin{proof}
Making use of the Stone-Weierstrass Theorem, there exists $\eta > 0$ such that, if $C$ is a $C^*$-algebra, $\pi:C_0((0,1]) \otimes C_0((0,1]) \to C$ is a $*$-homomorphism, and $w,x \in C$ are contractions, such that the norms of
\begin{gather}
\notag
[\pi(\id_{(0,1]} \otimes \id_{(0,1]}), w], \\
\notag
[\pi(\id_{(0,1]} \otimes \id_{(0,1]}^2), w], \\
\alabel{vCommuteDN-SWCommuteVerify}
[\pi(\id_{(0,1]} \otimes \id_{(0,1]}), x], \\
\notag
[\pi(\id_{(0,1]}^2 \otimes \id_{(0,1]}), x], \\
\notag
[\pi(\id_{(0,1]} \otimes \id_{(0,1]})x, w]
\end{gather}
are all less than $\eta$, then
\begin{equation}
\alabel{vCommuteDN-SWCommute}
 \|[w\pi(g \circ h \otimes g), x]\| < \frac\e{(m+1)}.
\end{equation}

We may assume that the elements of $\mathcal{F}$ are positive contractions.
Lemma \ref{GoodDnApproximations} gives us $\mathcal{G}$ and $\dl < \eta/3$ (which we use as our $\dl$) such that, given $e$ and an $(m+1)$-colourable c.p.\ approximation as in the hypothesis, we have
\begin{equation}
\alabel{vCommuteDN-Replacement}
\| \rho^{(i)}\sigma^{(i)}(e)\rho\sigma(a) - \rho^{(i)}\sigma^{(i)}(x) \| < \eta/3.
\end{equation}
for all $x \in \mathcal{F}$.

Fix $i \in \{0,\dots,m\}$ and $x \in \mathcal{F}$ for the moment.
Let $\hat \rho^{(i)}:C_0((0,1]) \tens F^{(i)} \to A$ be the $*$-homomorphism associated to $\rho^{(i)}$ as in Proposition \ref{OrderZeroStructure}.
We wish to use \eqref{vCommuteDN-SWCommute}, with $\pi$ coming by restricting $\hat \rho^{(i)}$ to $C_((0,1]) \otimes C^*(\sigma^{(i)}(e))$, $w=v_i$; so we must verify that the norms of the commutators in \eqref{vCommuteDN-SWCommuteVerify} are all sufficiently small.
We have
\[ \pi(\id_{(0,1] \otimes f} = \rho^{(i)}(f(\sigma^{(i)}(e))) \]
for any $f \in C_0((0,1])$, and we see that \eqref{vCommuteDN-LocalCommute} implies that the first two commutators of \eqref{vCommuteDN-SWCommuteVerify} have norm less than $\eta$.
For the third one, we have
\begin{align*}
 \pi(\id_{(0,1]} \otimes \id_{(0,1]})x &= \rho^{(i)}\sigma^{(i)}(e)x \\
&\approx^{\eqref{vCommuteDN-Replacement}}_{\eta/3} \rho^{(i)}\sigma^{(i)}(x) \\
&\approx_{\eta/3} x\pi(\id_{(0,1]} \otimes \id_{(0,1]}),
\end{align*}
the last line following by self-adjointness of $x$.
Using the last computation, we also get
\begin{align*}
 \pi(\id_{(0,1]}^2 \otimes \id_{(0,1]})x &\approx_{\eta/3} \rho^{(i)}(1_{F^{(i)}}\rho^{(i)}\sigma^{(i)}(x) \\
&= \hat\rho^{(i)}(\id_{(0,1]}^2 \otimes \sigma^{(i)}(x)) \\
&\approx_{\eta/3} x\pi(\id_{(0,1]}^2 \otimes \id_{(0,1]}),
\end{align*}
giving a bound on the fourth commutator in \eqref{vCommuteDN-SWCommuteVerify}.

For the last one, we have
\begin{align*}
\pi(\id_{(0,1]} \otimes \id_{(0,1]})xw &= \rho^{(i)}\sigma^{(i)}(e)xv_i \\
&\approx_{\eta/3} \rho^{(i)}\sigma^{(i)}(x)v_i \\
&\approx^{\eqref{vCommuteDN-LocalCommute}}_{\eta/2} v_i\rho^{(i)}\sigma^{(i)}(x) \\
&\approx_{\eta/3} w\pi(\id_{(0,1]} \otimes \id_{(0,1]})x.
\end{align*}

Therefore, we obtain \eqref{vCommuteDN-SWCommute}; note that $\pi(g \circ h \otimes g)w = g(h(\rho^{(i)})(\sigma^{(i)}(e)))v_i$, so that \eqref{vCommuteDN-SWCommute} becomes
\begin{equation}
\alabel{vCommuteDN-CommutePiece}
\| [v_ig(h(\rho^{(i)})(\sigma^{(i)}(e))),x]\| < \frac\e{(m+1)}.
\end{equation}
Thus we find, for $x \in \mathcal{F}$.
\begin{align*}
\|[v,x]\| &\leq \sum_{i=0}^m \|v_ig(h(\rho^{(i)})(\sigma^{(i)}(e))),x]\| \\
&<^\eqref{vCommuteDN-CommutePiece} (m+1)\frac\e{(m+1)},
\end{align*}
as required.
\end{proof}

\begin{thm}
\alabel{DefectEq}
Let $A$ be a separable, algebraically simple, exact $C^*$-algebra with strong tracial $\tilde m$-comparison and let $B \subseteq A$ be a separable subalgebra with nuclear dimension at most $m$.
Suppose that $\Phi:M_n \tens M_{m+1} \tens M_2\to A_\infty \cap B'$ is a c.p.c.\ order zero map which satisfies
\[ \tau\left(\Phi\left(1_{M_n \tens M_{m+1} \tens M_2}\right)\right) = 1 \]
for all $\tau \in T^1_\infty(A)$.
Then there exists a positive contraction $e \in (A_\infty)_+ \cap \Phi\left(M_n\right)'$ and $v \in A_\infty \cap B'$ such that $e$ acts as a unit on $B$,
%\[ v^*v = e-\Phi(1) \quad \text{and} \quad v = \Phi(e_{11})v. \]
\begin{equation}
\alabel{DefectEq-MainEq}
 v^*v = e-e\Phi\left(1_{M_n \tens M_{m+1} \tens M_2}\right) \quad \text{and} \quad v = \Phi\left(e_{11} \tens 1_{m+1} \tens 1_2\right)v.
\end{equation}

\end{thm}

\begin{proof}
Appealing to the diagonal sequence argument in Section \ref{DiagArgSect}, we will show that we can obtain $e \in A_\infty \cap \Phi(M_n)'$ which acts approximately as a unit on $B$ and $v \in A_\infty$ which approximately commutes with $B$, and such that \eqref{DefectEq-MainEq} holds approximately.
Therefore, let us fix a finite subset $\mathcal{F} \subset B$ and a tolerance $\e > 0$.
Set $\eta := (\e/4(m+1))^2$.
Using $h=g_{0,\eta}$ and $g(t) = t^{1/2}$ with Proposition \ref{vCommuteDN}, we obtain $\mathcal{G}$ and $\dl$; we may assume that $\dl < \e/2$.

Let $e \in B_+$ be a positive contraction such that $ex \approx_\dl x$ for all $x \in \mathcal{F}$.
Let $(F = F^{(0)} \dsum \cdots \dsum F^{(m)}, \sigma, \rho)$ be a $(m+1)$-decomposable c.p.\ approximation (for $B$) of $\mathcal{G} \cup \{e\}$ to within $\dl$.
By Lemma \ref{DefectEqCpc}, let $v_i \in A_\infty \cap \rho^{(i)}(F^{(i)})$ be such that
\begin{enumerate}
\item $v_i^*v_i = \rho^{(i)}(1_{F_i})\left(1-\Phi\left(1_{M_n \tens M_{m+1} \tens M_2}\right)\right)$; and
\item $v_i = \Phi\left(e_{11} \tens e_{ii} \tens 1_2\right)v_i$.
\end{enumerate}

Denote by $\hat\rho^{(i)}:C_0((0,1]) \tens F^{(i)} \to A$ the $*$-homomorphism associated to $\rho^{(i)}$ as in Proposition \ref{OrderZeroStructure}.
Then
\begin{align}
\notag
g_{0,\eta}(\rho^{(i)})(\sigma^{(i)}(e))^{1/2}\rho^{(i)}(1_{F^{(i)}})^{1/2} &=
\hat\rho^{(i)}(g_{0,\eta}^{1/2}\id_{(0,1]}^{1/2} \otimes \sigma^{(i)}(e)^{1/2}) \\
\notag
&\approx_{\e/4(m+1)} \hat\rho^{(i)}(\id_{(0,1]}^{1/2} \otimes \sigma^{(i)}(e)^{1/2}) \\
\alabel{DefectEq-CpMagic}
&= \rho^{(i)}\sigma^{(i)}(e)^{1/2}.
\end{align}

Now, define $v$ as in \eqref{vCommuteDN-vDef}, where $h=g_{0,\eta}$ and $g(t)=t^{1/2}$.
That $\|[v,a]\| \leq \e$ for $a \in \mathcal{F}$ is ensured by Proposition \ref{vCommuteDN}.
It is also clear from Lemma \ref{DefectEqCpc} \ref{DefectEqCpc-id} that
\begin{equation}
\alabel{DefectEq-UnitOnv}
v\Phi\left(e_{11} \tens 1_{m+1} \tens 1_2\right) = v
\end{equation}
holds exactly.
Finally,
\begin{align*}
v^*v &=^{\eqref{vCommuteDN-vDef}} \sum_{i,j} g_{0,\eta}(\rho^{(i)})(\sigma^{(i)}(e))^{1/2}v_i^*v_jg_{0,\eta}(\rho^{(j)})(\sigma^{(j)}(e))^{1/2} \\
&=^{\text{Lemma \ref{DefectEqCpc} \ref{DefectEqCpc-id}}} \sum_{i,j} g_{0,\eta}(\rho^{(i)})(\sigma^{(i)}(e))^{1/2}v_i^*\Phi\left(e_{11} \tens e_{ii} \tens 1_2\right)\\
&\qquad\qquad \Phi\left(e_{11} \tens e_{jj} \tens 1_2\right)v_jg_{0,\eta}(\rho^{(j)})(\sigma^{(j)}(e))^{1/2} \\
&=^{\text{Lemma \ref{DefectEqCpc} \ref{DefectEqCpc-id}}} \sum_{i} g_{0,\eta}(\rho^{(i)})(\sigma^{(i)}(e))^{1/2}v_i^*v_ig_{0,\eta}(\rho^{(i)})(\sigma^{(i)}(e))^{1/2} \\
&=^{\text{Lemma \ref{DefectEqCpc} \ref{DefectEqCpc-mvn}}} \sum_{i} g_{0,\eta}(\rho^{(i)})(\sigma^{(i)}(e))^{1/2}\rho^{(i)}(1_{F_i})^{1/2}\left(1-\Phi\left(1_{M_n \tens M_{m+1} \tens M_2}\right)\right)\\
&\qquad\qquad \rho^{(i)}(1_{F_i})^{1/2} g_{0,\eta}(\rho^{(i)})(\sigma^{(i)}(e))^{1/2} \\
&\approx_{\e/2}^\eqref{DefectEq-CpMagic} \sum_{i} (\rho^{(i)}\sigma^{(i)}(e))^{1/2}\left(1-\Phi\left(1_{M_n \tens M_{m+1} \tens M_2}\right)\right)(\rho^{(i)}\sigma^{(i)}(e))^{1/2} \\
&= \rho\sigma(e)\left(1-\Phi\left(1_{M_n \tens M_{m+1} \tens M_2}\right)\right) \\
&\approx_{\dl} e\left(1-\Phi\left(1_{M_n \tens M_{m+1} \tens M_2}\right)\right),
\end{align*}
as required.
\end{proof}

\begin{thm}
\alabel{MainThm}
Let $A$ be a separable, algebraically simple, exact $C^*$-algebra.
Suppose that one of the following hold.
\begin{enumerate}
\item \alabel{MainThm-dn} $A$ has finite nuclear dimension;
\item \alabel{MainThm-tcd} $A$ has locally finite nuclear dimension, strong tracial $m$-comparison and tracial $m$-almost divisibility for some $m$;
\item \alabel{MainThm-ctd} $A$ has locally finite nuclear dimension, $m$-comparison and tracial $m$-almost divisibility for some $m$;
\item \alabel{MainThm-cd} $A$ has locally finite nuclear dimension, $m$-comparison and $m$-almost divisibility for some $m$.
\end{enumerate}
Then $A$ is $\jsZ$-stable.
\end{thm}

\begin{proof}
By Theorem \ref{DimNucTrAD} and \cite{Robert:dimNucComp}, \ref{MainThm-dn} $\Rightarrow$ \ref{MainThm-ctd}.
By Proposition \ref{CompDivRelationships}, each of \ref{MainThm-ctd} and \ref{MainThm-cd} imply \ref{MainThm-tcd}.
Therefore, we shall show $\jsZ$-stability using \ref{MainThm-tcd}.

This must be proven by considering separately two cases (the first case is well-known but warrants restating for completeness).

In the case that $T^1(A) = \emptyset$, by \ccite{Theorem 1.2}{BlackadarCuntz:simple}, $A \tens \K$ contains a nonzero projection $p$.
Set $B:=p(A \tens \K)p$.
By strong tracial $m$-comparison, and since $T^1(A) = \emptyset$, we have $[a] \leq [b]$ for any nonzero $a,b \in M_\infty(B)_+$.
In particular, $B$ is purely infinite (see \cite{Cuntz:Ktheory}), and therefore by \ccite{Theorem 3.15}{KirchbergPhillips:ExactEmbedding}, $B$ is $\mathcal{O}_\infty$-stable.
But $\mathcal{O}_\infty$-stability is preserved within stable-isomorphism classes, and therefore $A$ is also $\mathcal{O}_\infty$-stable.
Finally, since $\mathcal{O}_\infty$ is $\jsZ$-stable, so is $A$.

Now we turn to the case where $T^1(A) \neq \emptyset$.
In light of Proposition \ref{LimitDStabilityCharacterization}, it suffices to find, for any $n$, a unital $*$-homomorphism
\[ \jsZ_{n-1,n} \to (A_\infty \cap A')/A^\perp. \]
By Proposition \ref{DimDropPresentation}, this is the same as finding a c.p.c.\ order zero map
\[ \Phi:M_n \to (A_\infty \cap A')/A^\perp \]
together with an element $v \in (A_\infty \cap A')/A^\perp$ such that
\[ v^*v = 1-\Phi\left(1_n\right) \quad \text{and} \quad v = \Phi\left(e_{11}\right)v. \]
By the diagonal sequence argument (see Section \ref{DiagArgSect}), it suffices to do this approximately, finding for a given finite subset $\mathcal{F} \subset A$ and a tolerance $\e > 0$ a c.p.c.\ order zero map $\Phi:M_n \to A_\infty$ and $v \in A_\infty$ such that:
\begin{enumerate}
\item \alabel{MainProof-comm1} $\|[\Phi(x),a]\| \leq \e\|x\|$ for all $x \in M_n$ and all $a \in \mathcal{F}$;
\item \alabel{MainProof-comm2} $\|[v,a]\| \leq \e$ for all $v \in M_n$ all $a \in \mathcal{F}$;
\item \alabel{MainProof-mv} $\left\|\left(v^*v-(1-\Phi(1))\right)a\right\| \leq \e$ for all $a \in \mathcal{F}$; and
\item \alabel{MainProof-id} $v = \Phi(e_{11})v$.
\end{enumerate}

Therefore, let us be given $\mathcal{F} \subset A$ finite and $\e > 0$.
Since $A$ has locally finite nuclear dimension, we may find $B \subseteq A$ such that $\dn B \leq m < \infty$ and for every $a \in \mathcal{F}$, there exists $a' \in B$ such that $a \approx_{\e/2} a'$.
By Theorem \ref{OneMap}, we may find a c.p.c.\ order zero map $\Phi:M_n \tens M_{m+1} \tens M_2 \to A_\infty \cap B'$ such that
\[ \tau\left(\Phi\left(1_{M_n \tens M_{m+1} \tens M_2}\right)\right) = 1 \]
for all $\tau \in T^1_\infty(A)$.
By Theorem \ref{DefectEq}, it follows that there exists $e \in (A_\infty)_+ \cap \overline\Phi(M_n)'$ and $v \in A_\infty \cap B'$ such that $e$ acts as a unit on $B$ and
\[ v^*v = e-e\overline\Phi\left(1_{M_n \tens M_{m+1} \tens M_2}\right) \quad \text{and} \quad v = \overline\Phi(e_{11} \tens 1 \tens 1)v. \]
Set $\Phi:= \overline\Phi\left(\,\cdot \tens 1_{m+1} \tens 1_2\right):M_n \to A_\infty$.
Let us verify that this $\Phi$ and $v$ approximately satisfies the relations (i.e.\ check \ref{MainProof-comm1}-\ref{MainProof-id}).
Briefly, the idea is that the relations are continuous, and that if $\mathcal{F}$ was contained exactly, rather than approximately, in $B$, then the the relations would hold exactly.

\textbf{\ref{MainProof-comm1} and \ref{MainProof-comm2} hold}: For a contraction $x \in M_n$ and any $a \in \mathcal{F}$, we have
\[ \Phi(x)a \approx_{\e/2} \Phi(x)a' = a'\Phi(x) \approx_{\e/2} a\Phi(x); \]
and likewise we can prove that $v$ approximately commutes with $\mathcal{F}$.

\textbf{\ref{MainProof-mv} holds}:
We have
\begin{align*}
\left\|\left(v^*v-\left(1-\Phi\left(1\right)\right)\right)a\right\| &\leq \left\|\left(v^*v-\left(1-\Phi\left(1\right)\right)\right)a'\right\| + 2\frac\e2 \\ 
&\leq \left\|\big(\left(1-\Phi\left(1\right)\right)e-\left(1-\Phi\left(1\right)\right)\big)a'\right\| + \e \\ 
&= \e,
\end{align*}
since $e$ acts as a unit on $B$.

Finally, that \ref{MainProof-id} holds is quite clear from the choice of $\Phi$.
\end{proof}

\begin{remark}
(i)
Regrettably, the author was unable to produce a proof of Theorem \ref{MainThm} at didn't involve a dichotomy between the tracial and traceless cases.
The approach that works in the tracial case fundamentally requires the construction of orthogonal elements in Lemma \ref{ManyOrthogTraceMaintainers}; in the traceless case, the construction is not guaranteed to produce nonzero elements, which proves problematic in further use of this lemma.
%However, a new idea resolving this dichotomy will appear in forthcoming work by the author, Leonel Robert, and Wilhelm Winter.

(ii) That separability is a necessary hypothesis in Theorem \ref{MainThm} is shown in \cite{NonZStable}.
\end{remark}

%Verify this part with W/L
The authors owes thanks to Leonel Robert for discussions that prompted the inclusion of the finitely-many-ideal case in the following corollary.

\begin{cor}
\alabel{DimNucZstable}
Let $A$ be a separable $C^*$-algebra with finite nuclear dimension
 and finitely many ideals.
Suppose that no quotient of an ideal of $A$ is elementary.
Then $A$ is $\jsZ$-stable.
\end{cor}

\begin{proof}
In the simple case, this is a consequence of Theorem \ref{MainThm}, Corollary \ref{ExistsAlgSimple}, and the fact that $\jsZ$-stability and the value of nuclear dimension are constant under stable isomorphism classes.

The general case is by induction, since if $A$ has finitely many ideals then $A$ can be expressed as an extension
\[ 0 \to I \to A \to B \to 0, \]
where $B$ is simple and $I$ has fewer ideals than $A$.
$I$ and $B$ have finite nuclear dimension if $A$ does \ccite{Proposition 2.3 (iv) and Proposition 2.5}{WinterZacharias:NucDim}.
By induction, they are $\jsZ$-stable, and so by \ccite{Theorem 4.3}{TomsWinter:ssa}, $A$ is $\jsZ$-stable.
\end{proof}

\begin{cor}
\alabel{CuZstable}
Let $A$ be a separable, simple $C^*$-algebra with locally finite nuclear dimension.
The following are equivalent.
\begin{enumerate}
\item \alabel{CuZstable-alg} $A \iso A \tens \jsZ$;
\item \alabel{CuZstable-W} $W(A) \iso W(A \tens \jsZ)$;
\item \alabel{CuZstable-Cu} $\Cu(A) \iso \Cu(A \tens \jsZ)$.
\end{enumerate}
\end{cor}

\begin{proof}
\ref{CuZstable-alg} $\Rightarrow$ \ref{CuZstable-W} and \ref{CuZstable-Cu} is obvious.

\ref{CuZstable-W} $\Rightarrow$ \ref{CuZstable-alg}:
The proof of \ccite{Proposition 3.7}{Winter:pure} does not at all use the fact that $A$ is unital, and thereby shows that $W(A \tens \jsZ)$ has $0$-comparison and $0$-almost divisible Cuntz semigroup (properties which, we may recall, do not require algebraic simplicity).

By Corollary \ref{ExistsAlgSimple}, let $B$ be a nonzero hereditary subalgebra of $A$ which is algebraically simple.
Then $W(B) \subset W(A) \subset \Cu(A) = \Cu(B)$.
Since $W(A)$ has $0$-comparison, so does $W(B)$.
It follows from \ccite{Theorem 4.4.1}{BRTTW} that $W(B)$ is hereditary in $\Cu(B)$, and therefore also in $W(A)$.
Thus, $0$-almost divisibility for $W(A)$ implies $0$-almost divisibility for $W(B)$.
Theorem \ref{MainThm} then shows that $A \iso A \tens \jsZ$.

\ref{CuZstable-Cu} $\Rightarrow$ \ref{CuZstable-alg}: $\Cu(A) \iso W(A \tens \K)$, so by \ref{CuZstable-W} $\Rightarrow$ \ref{CuZstable-alg}, we see that \ref{CuZstable-Cu} implies that $A \tens \K$ is $\jsZ$-stable.
Since $\jsZ$-stability passes to hereditary subalgebras, $A \iso A \tens \jsZ$.
\end{proof}

\section{Approximately subhomogeneous $C^*$-algebras}
\alabel{ASHSec}
An important consequence of Theorem \ref{MainThm} is a characterization of slow dimension growth for simple approximately subhomogeneous $C^*$-algebras.
Roughly, slow dimension growth indicates that the algebra has a system for which the ratio of the topological to matricial dimension vanishes in the limit.
We refer to \ccite{Definition 5.3}{ProjlessReg} for a precise definition (which is trickier to produce in the nonunital, as opposed to unital, case).
In the unital case, the characterization was shown in \ccite{Corollary 7.5}{Winter:pure} and \ccite{Corollary 1.3}{Toms:rigidity}.

\begin{comment}
%Can't seem to make the following proof work
\TODO{define locally subhomogeneous}

\begin{lemma}
\alabel{LocalSHHered}
Let $A$ be a locally subhomogeneous $C^*$-algebra and $B \subset A$ a hereditary subalgebra.
Then $B$ is locally subhomogeneous.
\end{lemma}

\begin{proof}
Let $\mathcal{F} \subset B$ be a finite subset and let $\e > 0$.
We must show that there exists a subhomogeneous $C^*$-algebra $C \subseteq B$ such that $\mathcal{F} \subset_\e C$.

Set $\eta = $ \TODO{}.
Let $e \in B_+$ be such that $ex \approx_\eta x \approx_\eta xe$ for all $x \in \mathcal{F}$.
We may find $e' \in A_+$ such that $e' \approx_\eta e$ and $e'$ is contained in a subhomogeneous $C^*$-algebra $C' \subseteq A$.
It follows, of course, that $\overline{e'C'e'}$ is itself subhomogeneous.

Since $e' \approx_\eta e$, \TODO{cite} says that there exists $x'$
\end{proof}
\end{comment}

\begin{lemma}
\alabel{ExistsAlgSimpleASH}
Let $A$ be a simple approximately subhomogeneous algebra.
Then there exists a nonzero hereditary subalgebra $B$ of $A$ such that $B$ is algebraically simple and approximately subhomogeneous.
In particular, $A$ is stably isomorphic to $B$.

Moreover, if $A$ has slow dimension growth as in \ccite{Definition 5.3}{ProjlessReg} then $B$ can be chosen to also have slow dimension growth.
\end{lemma}

\begin{proof}
We shall produce $A$ quite as in the proof of Corollary \ref{ExistsAlgSimple}, except that the element $b \in \Ped(A)$ is chosen with some care, to ensure that $B$ is approximately subhomogeneous.
Let $A$ be the closed union of an increasing sequence of subhomogeneous subalgebras,
\[ A_1 \subseteq A_2 \subseteq \cdots. \]
Now, take $b \in \Ped(A_1) \subseteq \Ped(A)$.
As in the proof of Corollary \ref{ExistsAlgSimple}, $A$ is stably isomorphic to $B:=\overline{bAb}$, which is algebraically simple.
Moreover,
\[ B = \overline{bAb} = \overline{\bigcup \overline{bA_ib}}, \]
and since $A_i$ is subhomogeneous, so is $\overline{bA_ib}$, for each $i$.

If $A$ has slow dimension growth then we may pick the inductive sequence
\[ A_1 \subseteq A_2 \subseteq \cdots, \]
to witness this, and then it follows (cf.\ \ccite{Proposition 5.2}{ProjlessReg}) that
\[ \overline{bA_1b} \subseteq \overline{bA_2b} \subseteq \cdots \]
witnesses slow dimension growth for $B$.
\end{proof}

\begin{cor}
Let $A$ be a simple approximately subhomogeneous algebra.
Then $A$ has slow dimension growth, as defined in \ccite{Definition 5.3}{ProjlessReg}, if and only if $A$ is $\jsZ$-stable.
\end{cor}

\begin{proof}
($\Rightarrow$):
By Lemma \ref{ExistsAlgSimple}, we may assume without loss of generality that $A$ is algebraically simple.
By \ccite{Corollary 5.9}{ProjlessReg}, $\Cu(A)$ has $0$-comparison; since there is an order embedding $W(A) \subset \Cu(A)$, it follows that $W(A)$ also has $0$-comparison.
By \ccite{Corollary 7.2}{TT:ranks}, $\Cu(A)$ has $0$-almost divisibility; then by Proposition \ref{CompDivRelationships} \ref{CDR-CuDiv}, it follows that $A$ has $0$-almost divisibility.

Therefore, $B$ satisfies hypothesis \ref{MainThm-cd} of Theorem \ref{MainThm}, whence $B$ is $\jsZ$-stable.

($\Leftarrow$): 
Suppose $A \iso A \tens \jsZ$.
By \ccite{Proposition 3.6}{ProjlessReg}, there exists an inductive system
\[ A_1 \to A_2 \to \cdots \]
whose limit is $A$, such that $A_i$ is recursive subhomogeneous with finite dimensional total space and compact spectrum, and the connecting maps are injective and full.
On the other hand, $\jsZ$ is the inductive limit of a system
\[ \jsZ_{p_1-1,p_1} \to \jsZ_{p_2-1,p_2} \to \cdots, \]
for some $p_i \to \infty$.
In fact, by taking a tail, we may make $p_i$ as large as we want, so that the ratio of the dimension of the total space of $A_i$ to $p_i$ goes to $0$ as $i \to \infty$.
Then it follows that
\[ A_1 \tens \jsZ_{p_1-1,p_1} \to A_2 \tens \jsZ_{p_2-1,p_2} \to \cdots \]
is a system whose limit is $A \tens \jsZ$, and witnesses that this algebra has slow dimension growth.
\end{proof}

\end{document}